\documentclass{scrartcl}

\usepackage{lmodern}
\usepackage[utf8]{inputenc}
\usepackage[T1]{fontenc}
\usepackage[english]{babel,csquotes}

\usepackage[backend=biber,style=alphabetic,url=false,giveninits=true,sorting=anyt]{biblatex} 
\usepackage{graphicx}

\usepackage{amsmath,amssymb,amsthm}
\usepackage{mathtools,thmtools,stmaryrd}
\usepackage{mathrsfs}
\usepackage{calc}

\usepackage[colorlinks=true]{hyperref}
\usepackage[shortlabels]{enumitem}
\usepackage[nameinlink]{cleveref}
\crefname{equation}{}{}
\crefname{enumi}{}{}

\newcommand{\N}{\mathbb{N}}

\newcommand{\R}{\mathbb{R}}

\newcommand{\cone}{\mathbin{\times\!\!\!\!\times}}
\newcommand{\mres}{\llcorner}
\newcommand{\eps}{\varepsilon}
\newcommand{\haus}{{\mathscr H}}
\newcommand{\mass}{\mathbb{M}}
\newcommand{\ener}{{\mathcal E}}
\newcommand{\one}{\mathbf{1}}
\newcommand{\dd}{\mathop{}\mathopen{}\mathrm{d}}
\DeclarePairedDelimiter\slice{\langle}{\rangle}
\DeclarePairedDelimiterX\islice[1]{\langle}{\rangle}{\!\left\langle#1\right\rangle\!}
\DeclarePairedDelimiter\abs{\lvert}{\rvert}
\DeclarePairedDelimiter\norm{\lVert}{\rVert}
\DeclarePairedDelimiter\set{\{}{\}}
\DeclarePairedDelimiter\rectcurr{\llbracket}{\rrbracket}
\DeclareMathOperator{\sign}{sign}
\DeclareMathOperator{\spn}{span}
\DeclareMathOperator{\tanspace}{Tan}
\DeclareMathOperator{\lipconst}{Lip}

\DeclareMathOperator{\dvg}{div}
\DeclareMathOperator{\img}{Img}
\DeclareMathOperator{\diam}{diam}

\newcommand{\tplan}[1]{\mathbf{#1}}

\newcommand{\lipone}{{\mathrm{Lip}_1}}
\newcommand{\measspace}{\mathscr{M}}

\newcommand{\tplanset}{\mathbf{TP}}
\newcommand{\otplanset}{\mathbf{OTP}}
\newcommand{\tpathset}{\mathit{TP}}
\newcommand{\otpathset}{\mathit{OTP}}

\newcommand{\weakstarto}{\xrightharpoonup{\star}}
\newcommand{\xstrongto}[2][]{\xrightarrow[#1]{#2}}

\newcommand{\xweakstarto}[1]{\xrightharpoonup[#1]{\star}}

\newcommand{\refabove}[2]{\stackrel{\mbox{\normalfont\tiny #1}}{#2}}
\newlength\oversetwidth
\newlength\underwidth
\newcommand\alignedrefabove[2]{
  \settowidth\oversetwidth{$\overset{#1}{#2}$}
  \settowidth\underwidth{$#2$}
  \setlength\oversetwidth{\oversetwidth-\underwidth}
  \hspace{.5\oversetwidth}
  &\,
  \settowidth\oversetwidth{$\overset{#1}{#2}$}
  \settowidth\underwidth{$#2$}
  \setlength\oversetwidth{\oversetwidth-\underwidth}
  \hspace{-.5\oversetwidth}
  \overset{\mbox{\normalfont\tiny #1}}{#2}
}

\newenvironment{innerproof}
 {\proof}
 {\endproof}

\numberwithin{equation}{section}

\declaretheorem[name=Theorem,within=section]{thm}
\declaretheorem[name=Lemma,numberlike=thm]{lem}
\declaretheorem[name=Proposition,numberlike=thm]{prp}

\declaretheorem[name=Definition,numberlike=thm,style=definition]{dfn}

\declaretheorem[name=Remark,numberlike=thm,style=remark]{rmk}

\title{Stability of optimal traffic plans in the irrigation problem}
\author{Maria Colombo, Antonio De Rosa, Andrea Marchese,\\ Paul Pegon, Antoine Prouff}
\date\today

\begin{document}

\maketitle

\begin{abstract}
We prove the stability of optimal traffic plans in branched transport. In particular, we show that any limit of optimal traffic plans is optimal as well. This is the Lagrangian counterpart of the recent Eulerian version proved in \cite{ColDeRMar19-2}.
\end{abstract}

\section{Introduction}\label{sec:intro}
Given two nonnegative and finite Borel measures $\mu^-,\mu^+$ on $\R^d$ of equal total mass, the irrigation problem consists in connecting $\mu^-$ to $\mu^+$ with minimal cost, where in branched transport the displacement is performed on a $1$-dimensional network and the transport cost for a collection of particles of total mass $m$ travelling a distance $\ell$ along a common stretch is proportional to $\ell \times m^\alpha$, for a fixed parameter $\alpha \in (0,1)$. This problem may be cast in two main statical frameworks: an Eulerian one \cite{Xia03}, based on vector valued measures (more precisely normal $1$-currents) called \emph{transport paths}, and a Lagrangian one \cite{MadMorSol03, BerCasMor05}, based on positive measures on a set of curves (or trajectories) called \emph{traffic plans}. We refer to the book \cite{BCM} for the general theory of branched transport, and to the first sections of the more recent works \cite{Peg17,ColDeRMar18,ColDeRMar19-1} and the references therein.\\

In this paper, we tackle the question of the \emph{stability} in the Lagrangian framework: if $\{\mu_n^-\}_{n \in \N}$ and $\{\mu_n^+\}_{n \in \N}$ converge respectively to $\mu^-$ and $\mu^+$, and if $\{\tplan P_n\}_{n \in \N}$ is a sequence of optimal traffic plans for the marginals $(\mu_n^-,\mu_n^+)$, converging to a traffic plan $\tplan P$, is it true that $\tplan P$ is optimal for $(\mu^-,\mu^+)$? The positive answer is classically known above the critical threshold $\alpha > 1 - 1/d$ both for the Lagrangian and the Eulerian formulation. A positive answer for every $\alpha \in (0,1)$ has been recently given for the Eulerian formulation in \cite{ColDeRMar19-2}. Although the Eulerian and Lagrangian problems are essentially equivalent (see \cite{PaoSte06,Peg17}), the Eulerian viewpoint carries less information than the Lagrangian one, and the Lagrangian stability is not a straightforward consequence of the Eulerian one.

\paragraph{Main result}

Denote by $\otplanset(\mu^-,\mu^+)$ the set of optimal traffic plans with marginals $(\mu^-,\mu^+)$. Modulo some technical assumptions (necessary to the validity of the statement), we prove the following result. See \Cref{mainthm} for the correct statement.
\begin{thm}[Short statement]\label{mainthmrough}
Let $\mu_n^\pm \weakstarto \mu^\pm$ and let $\tplan P_n \in \otplanset(\mu_n^-, \mu_n^+)$ and assume that $\tplan{P}_n$ converges to $\tplan{P}$. Then, up to mild technical assumptions, $\tplan P \in \otplanset(\mu^-, \mu^+)$.
\end{thm}

\paragraph{Strategy of the proof}
Our proof relies on the general stability result proved for the Eulerian setting in \cite[Theorem~1.1]{ColDeRMar19-2}. The classical way to associate to a Lagrangian traffic plan $\tplan P$ an Eulerian transport path $T=T_{\tplan P}$ consists in integrating (w.r.t. $\tplan P$) the obvious vector measures associated to the curves supporting $\tplan P$. The two suitably defined notions of transportation cost coincide on optimizers.

Taking $\tplan P$, $\{\tplan P_n\}_{n\in\N}$ as in \Cref{mainthmrough} we consider the induced transport paths $T$, $\{T_n\}_{n\in\N}$. One can easily show that $T$ and $\{T_n\}_{n\in\N}$ satisfy the hypotheses of \cite[Theorem~1.1]{ColDeRMar19-2}, so that $T$ is an optimal transport path for the marginals $(\mu^-,\mu^+)$. Nevertheless in principle it could happen that the cost of $T$ as a transport path and the cost of $\tplan P$ as a traffic plan do not coincide. This possibility can be attributed only to a specific phenomenon: some curves of $\tplan P$ partially overlap with opposite orientations, thus producing \emph{cancellations} at the level of vector measures. Most of our work consists in excluding the occurrence of such phenomenon.\\

The article closely follows the structure of the proof. After setting the notation, main definitions and preliminary results in \Cref{sec:prelim}, we argue by contradiction assuming that $\tplan P$ produces cancellations at the Eulerian level. \Cref{sec:canceltocycles} provides existence of \enquote{many Lagrangian cycles} in $\tplan P$, i.e. many pairs of distinct points $(x,y)$ such that both the family of those trajectories crossing $x$ after $y$ and those crossing $y$ after $x$ have positive measure according to $\tplan P$. From this, we deduce in \Cref{sec:cyclesandquasi} the existence of \enquote{quasi-cycles} in the $\tplan P_n$'s, roughly saying that for any such pair $(x,y)$ a certain amount of trajectories passes arbitrarily close to $x$ and $y$ in both orders, for $n$ large enough. In \Cref{sec:mainproof} we show that this leads to a contradiction by constructing a better competitor for $\tplan P_n$, removing portions of such trajectories, thus completing the proof of \Cref{mainthmrough}.

\section{Preliminaries}\label{sec:prelim}

In this section, we gather some definitions and basic facts that will be used throughout the paper. The notation is mostly consistent with \cite{ColDeRMar19-1}. 

\subsection{Background, notation, and main result}
We denote by $\abs{x}$ the Euclidean norm of $x\in \R^d$ and by $B_r(x), \bar B_r(x)$ respectively the open and the closed ball with center $x$ and radius $r$.
From now on we fix, $\alpha \in (0,1)$, $R > 0$ and $X := \bar B_R(0) \subseteq \R^d$. 
Except for the obvious cases, the measures that we consider are always Radon measures. Here is a list of notation used throughout the paper:
\begin{itemize}\leftskip 1.8 cm\labelsep=.6 cm
\setlength{\itemsep}{0pt}
\item[$\one_A$]  indicator function of a set $A$ valued in $\{0,1\}$
\item[$d(x,A)$]  $\coloneqq \inf_{y\in A} \abs{x-y}$, distance between the point $x$ and the set $A$
\item[$\measspace^k(Y)$]  space of finite (signed or vector) Borel measures on $Y$ valued in $\R^k$
\item[$\measspace^1_+(Y)$]  set of nonnegative finite Borel measures on $Y$
\item[$\mu_n \weakstarto \mu$]  weak-$\star$ convergence of measures in the duality between $C^0(Y,\R^k)$ and $\measspace^k(Y)$ when $Y$ is compact, i.e. $\int f \dd\mu_n \to \int f\dd\mu$ for every $f\in C^0(Y,\R^k)$
\item[$\mu \mres A$]  $\coloneqq \one_A \mu$, restriction of the measure $\mu$ to the subset $A$
\item[$f_\sharp \mu$]  push-forward of the measure $\mu$ on $Y$ by the map $f : Y \to Y'$, i.e. $f_\sharp\mu(A) \coloneqq \mu(f^{-1}(A))$
\item[$f\mu$]  (vector) measure defined by $[f\mu](A) \coloneqq \int_A f\dd\mu$ for every Borel set $A$, when $\mu$ is a nonnegative Borel measure and $f$ a Borel (vector-valued) map such that $\int \abs{f} \dd\mu < +\infty$
\item[$\mass(\mu)$]  mass of the measure $\mu$
\item[$\mass^\alpha(\mu)$]  $\coloneqq \sum_{x\in Y} |\mu(\{x\})|^\alpha$ when $\alpha \in [0,1)$ and $\mu \in \measspace^1(Y)$ is atomic, set to $+\infty$ if $\mu$ is not atomic
\item[$\mu \leq \nu$]  means that $\mu(A) \leq \nu(A)$ for all Borel set $A$
\item[$\norm{f}_\infty$]  $\coloneqq \sup_{x\in Y} \abs{f(x)}$ supremum norm of $f : Y \to \R^k$
\item[$\haus^k$]  $k$-dimensional Hausdorff measure
\item[$\haus^k_\delta$] $k$-dimensional Hausdorff pre-measures (see \cite[Definition 2.1]{EG})
\item[$\lipone$]  set of $1$-Lipschitz curves $\gamma : \R_+ \to X$, endowed with the (compact and metrizable) topology of uniform convergence on compact subsets of $\R_+$
\item[$\img \gamma$]  image $\gamma(I)$ of a curve $\gamma : I\subseteq \R \to \R^d$
\item[$T_\infty(\gamma)$]  $\coloneqq \inf \{t \in \R_+ : \gamma \text{ is constant on } [t,+\infty)\} \in [0,\infty]$, stopping time of $\gamma$
\item[$\gamma_{|[a,b]}$]  restriction of $\gamma\in \lipone$ to an interval $[a,b] \subseteq \R_+$ defined by $t\mapsto\gamma(t + a)$ for $t \in [0, b - a]$ and $t\mapsto\gamma(b)$ for $t \ge b - a$
\item[$e_0,e_\infty$]  evaluation maps $\gamma \mapsto \gamma(0)$ and $\gamma \mapsto \gamma(\infty) \coloneqq \lim_{t\to+\infty}\gamma(t)$ when $\gamma$ has finite length
\item[$\tanspace(x,E)$]  tangent space line at point $x$ to $E$ when $E$ is $1$-rectifiable, i.e. it is contained in a countable union of images of Lipschitz curves up to an $\haus^1$-null set; it is $\haus^1$-a.e. defined on $E$ (see \cite[Definition 2.86]{afp}).
\end{itemize}


A \emph{traffic plan} $\tplan P$ is a measure in $\measspace^1_+(\lipone)$ such that $\int_\lipone T_\infty \dd\tplan P < \infty$. If there exists a $1$-rectifiable set $E$ such that
\begin{equation}\label{eq:defrectif}
\haus^1(\img \gamma \setminus E) = 0\text{ for $\tplan P$-almost every }\gamma \in \lipone,
\end{equation}
then $\tplan P$ is said rectifiable. We list the main objects that we need regarding traffic plans:
\begin{itemize}\leftskip 1.8 cm\labelsep=.6 cm
\setlength{\itemsep}{0pt}
\item[$\tplanset$]  space of traffic plans 
\item[$\tplanset(\mu^-,\mu^+)$]  set of traffic plans $\tplan P$ such that $(e_0)_\sharp \tplan P = \mu^-$, $(e_\infty)_\sharp \tplan P = \mu^+$
\item[$\tplan P_n \weakstarto \tplan P$]  weak-$\star$ convergence in $\measspace^1(\lipone)$
\item[$\theta_{\tplan P}(x)$]  $\coloneqq \tplan P(\{\gamma \in \lipone : x\in \img \gamma\})$, multiplicity at $x$ w.r.t. $\tplan P$
\item[$\Theta_{\tplan P}(x)$]  $\coloneqq \int_\lipone \haus^0(\gamma^{-1}(x))\dd\tplan P$, full multiplicity at $x$ w.r.t. $\tplan P$
\item[$\ener^\alpha(\tplan P)$]  $\coloneqq \int_\lipone \int_{\R_+} \theta_\tplan P(\gamma(t))^{\alpha - 1} |\gamma'(t)| \dd{t} \dd{\tplan P}(\gamma)$, $\alpha$-energy of $\tplan P$
\item[$\otplanset(\mu^-,\mu^+)$]  set of optimal traffic plans with marginals $(\mu^-,\mu^+)$, that is $\ener^\alpha(\tplan P) < +\infty$ and $\ener^\alpha(\tplan P) \leq \ener^\alpha(\tplan Q)$ for every $\tplan Q \in \tplanset(\mu^-,\mu^+)$
\item[$\Sigma_{\tplan P}$]  $\coloneqq \{x : \theta_{\tplan P}(x) > 0\}$ network associated with $\tplan P$; it is $1$-rectifiable (by \cite[Section 2.1]{Peg17} or \cite[Lemma~6.3]{BerCasMor05}), and when $\tplan P$ is rectifiable then \cref{eq:defrectif} holds with $E= \Sigma_{\tplan P}$ .
\end{itemize}

We can now state the correct version of \Cref{mainthmrough}.

\begin{thm}[Stability of optimal traffic plans in the irrigation problem] \label{mainthm}
Let $\alpha \in (0, 1)$, $\mu^-$, $\mu^+$ be mutually singular positive finite measures on $\bar{B}_R(0) \subseteq \R^d$, $R > 0$, satisfying $\mu^-(\R^d) = \mu^+(\R^d)$. Let $\{\mu_n^-\}_{n \in \N}$ and $\{\mu_n^+\}_{n \in \N}$ be sequences of positive finite measures on $\bar{B}_R(0)$ such that $\mu_n^-(\R^d) = \mu_n^+(\R^d)$ for every $n \in \N$ and
\[
\mu_n^\pm \weakstarto \mu^\pm,
\]
and assume there exist  $\tplan P_n \in \otplanset(\mu_n^-, \mu_n^+)$ satisfying
\[
\sup_{n \in \N} \left\{ \ener^\alpha(\tplan P_n) + \int_\lipone T_\infty(\gamma) \dd{\tplan P_n}(\gamma) \right\} < \infty,
\]
and
\[
\tplan{P}_n \xweakstarto{n\to\infty} \tplan{P},
\]
for some $\tplan P$.
Then $\tplan P\in\otplanset(\mu^-, \mu^+)$.
\end{thm}

\subsection{Transport paths}

A \emph{transport path} $T$ over $X$ is a normal $1$-current, or equivalently a vector measure on $X$ whose distributional divergence is a signed measure. Let us summarize the classical notation for transport paths in the following table:
\begin{itemize}\leftskip 1.8 cm\labelsep=.6 cm
\setlength{\itemsep}{0pt}
\item[$\tpathset$]  space of transport paths 
\item[$\partial T$]  $\coloneqq - \dvg T$ where $\dvg T$ is the distributional divergence of $T$ on $\R^d$
\item[$\tpathset(\mu^-,\mu^+)$]  set of transport paths $T$ such that $\partial T = \mu^+- \mu^-$
\item[$T_n \weakstarto T$]  weak-$\star$ convergence in $\measspace^d(X)$
\item[$\rectcurr{E,\vec \theta}$]  $\coloneqq \vec \theta \haus^1\mres E$ when $E$ is $1$-rectifiable and $\vec\theta : E \to \R^d$ is such that $\vec \theta(x) \in \tanspace(x,E)$ for $\haus^1$-a.e. $x$, $\int_E \abs{\vec\theta}\dd\haus^1 <\infty$, and $\vec \theta \haus^1\mres E$ is normal ; transport paths of this form are called rectifiable
\item[$\mass^\alpha(T)$]  defined by $\int_E \abs{\vec\theta}^{\alpha}\dd\haus^1$ for $T = \rectcurr{E,\vec\theta}$, set to $+\infty$ if $T$ is not rectifiable
\item[$\otpathset(\mu^-,\mu^+)$]  set of optimal transport paths $T \in \tpathset(\mu^-,\mu^+)$, meaning that $\mass^\alpha(T) < +\infty$ and $\mass^\alpha(T) \leq \mass^\alpha(S)$ for every $S\in \tpathset(\mu^-,\mu^+)$
\item[$I_\gamma$]  transport path induced by the curve of finite length $\gamma \in\lipone$ and defined by $\langle I_\gamma,\omega \rangle \coloneqq \int_{\R_+} \omega(\gamma(t))\cdot \gamma'(t) \dd t$ for every $\omega \in C^\infty_c(X,\R^d)$ ; its boundary is $\partial I_\gamma = \delta_{\gamma(\infty)}-\delta_{\gamma(0)}$
\item[$T_{\tplan P}$]  $\coloneqq \int_\lipone I_\gamma \dd\tplan P(\gamma)$, transport path induced by $\tplan P$; its boundary is $\partial T_{\tplan P} = (e_\infty)_\sharp \tplan P -(e_0)_\sharp \tplan P$.
\end{itemize}

In the last definition, the integration should be intended in the following sense. Let $I$ be a finite measure space and for every $t\in I$ let $\mu_t$ be a measure on $\R^n$, possibly real- or vector-valued, such that 
 $t\mapsto \mu_t(E)$ is measurable for every Borel set $E$ in $\R^n$; 
%
the integral $\int_I \mathbb M(\mu_t) \dd t$ is finite.
%
Then we denote by $\int_I \mu_t \dd t$ the measure on 
$\R^n$ defined by
\begin{equation}
\label{int:meas}
{\textstyle \big[ \int_I \mu_t \dd t \big]}(E)
:= \int_I \mu_t(E) \dd t
\quad\text{for every Borel set $E$ in $\R^n$.}
\end{equation}

When $T = T_{\tplan P}$ we say that \emph{$\tplan P$ decomposes $T$}. Following \cite{ColDeRMar18}, a \emph{good decomposition}, first introduced by Smirnov (see \cite[Section 1.2]{smirnov}) for normal currents, is a decomposition where neither cycles nor cancellations occur:
\begin{dfn}[Good decomposition]\label{gooddecompo}
Let $T$ and $\tplan P$ be a transport path and traffic plan such that $T = T_{\tplan P}$. Then $\tplan{P}$ is said to be a good decomposition of $T$ if:
\begin{enumerate}[(A)]
\item $\tplan{P}$ is supported on nonconstant simple curves; \label{prop:A}
\item $\mass(T) = \int_\lipone \mass(I_\gamma) \dd{\tplan{P}}(\gamma)$; \label{prop:B}
\item $\mass(\partial T) = \int_\lipone \mass(\partial I_\gamma) \dd{\tplan{P}}(\gamma)=2\tplan{P}(\lipone)$. \label{prop:C}
\end{enumerate}
\end{dfn}
According to the Decomposition Theorem of Smirnov \cite[Theorem C]{smirnov} (see also \cite{San14} for a Dacorogna-Moser approach), any acyclic transport path, hence any optimal transport path, admits a good decomposition.\\

\subsection{On curves and rectifiability} \label{subsec:rect}

Here we collect some basic results about $1$-Lipschitz curves, $1$-rectifiable sets and rectifiable traffic plans.


\begin{dfn} Let $\gamma \in \lipone$ of finite length and $\tplan P \in \tplanset$ a rectifiable traffic plan. We say that:
\begin{itemize}
\item $x$ is a \emph{regular point} of $\gamma$ if $x \notin \{\gamma(0),\gamma(\infty)\}$,  $\tanspace(x,\img \gamma)$ exists, $\gamma^{-1}(x)$ is finite and for all preimage $t$ of $x$, $\gamma'(t)$ exists and spans $\tanspace(x,\img{\gamma})$. Notice that $\gamma'(t)$ might have different orientations for different preimages of $x$;
\item $x$ is a \emph{regular point} of $\tplan P$ if $\tanspace(x,\Sigma_{\tplan P})$ exists and if for $\tplan P$-a.e. curve $\gamma$, $x$ is a regular point of $\gamma$ such that $\tanspace(x,\Sigma_{\tplan P}) = \tanspace(x,\img \gamma)$.
\end{itemize}
\end{dfn}
\begin{rmk}\label{aeregul}
As a direct consequence of the Area formula (\cite[Section 3.3]{EG}) and the definition of the tangent space (see \cite[Section~2.11]{afp} or \cite[212--213]{Mat}), we get that $\haus^1$-a.e. $x\in \img \gamma$ is a regular point of $\gamma$. Using the rectifiability of $\tplan P$ and Fubini's Theorem, we immediately deduce that $\haus^1$-a.e. point $x\in \Sigma_{\tplan P}$ is regular for $\tplan P$. Indeed, denoting $f:\lipone \times X \to \mathbb R$ such that $f(\gamma,x)=0$ if $\gamma$ has finite length and $x$ is a regular point of $\gamma$ and $f(\gamma,x)=1$ otherwise, it holds
$$0=\int_{\lipone}\int_{X}f(\gamma,x)\dd\haus^1 \mres \Sigma_{\tplan P} \dd{\tplan P}=\int_{X}\int_{\lipone}f(\gamma,x) \dd{\tplan P}\dd\haus^1 \mres \Sigma_{\tplan P}.$$
\end{rmk}
Now at each regular point $x$ of $\gamma \in \lipone$, we define:
\begin{equation}
\vec m_\gamma(x) \coloneqq \sum_{t \in \gamma^{-1}(x)} \gamma'(t)/|\gamma'(t)|,
\end{equation}
and at each regular point $x$ of $\tplan P$:
\begin{equation}
\vec \theta_{\tplan P}(x) \coloneqq \int_\lipone \vec m_\gamma(x) \dd\tplan P(\gamma).
\end{equation}
Both are well-defined $\haus^1$-a.e. respectively on $\img \gamma$ and $\Sigma_{\tplan P}$, and set to $0$ outside. Notice that by definition $\vec m_\gamma(x) \in \tanspace(x,\img\gamma)$ (with integer norm) for $\haus^1$-a.e. $x\in \img \gamma$ and $\vec \theta_{\tplan P}(x) \in \tanspace(x,\Sigma_{\tplan P})$ for $\haus^1$-a.e. $x\in\Sigma_{\tplan P}$. A direct use of the Area Formula and Fubini's Theorem yields:
\begin{align}
I_\gamma &= \rectcurr{\img \gamma, \vec m_\gamma},&
T_{\tplan P} &= \rectcurr{\Sigma_{\tplan P}, \vec \theta_{\tplan P}}.
\end{align}
Following \cite[Definition 11.9]{Mat}, given a $1$-dimensional linear subspace $V \subseteq \R^d$, $x \in \R^d$, $s \in (0, 1)$ and $r \in [0, \infty]$, we define the two-sided cone:
\[
X(x, r, V, s) := \set{y \in \R^d : d(y - x, V) \le s |y - x|} \cap \bar B_r(x),
\]
and for any nonzero vector $v \in \R^d$, we define the one-sided cone:
\[
X_\pm(x, r, v, s) := X(x, r, \spn v, s) \cap \set{y \in \R^d : \pm v \cdot (y-x) \ge 0}.
\]

\begin{dfn}[Proper crossing] \label{def:regularity}
Consider a cone $X(x_0, r, V, s)$ and a curve $\gamma \in \lipone$ such that $\gamma(t_0) = x_0$, and $\gamma'(t_0)$ exists and spans $V$. We say that $\gamma$ \emph{crosses the cone properly} at time $t_0 \in (0,T_\infty(\gamma))$ if there exist $t_\textrm{in} < t_0 < t_\textrm{out}$ such that
\begin{enumerate}[(i)]
\item $\gamma([t_\textrm{in}, t_0]) \subseteq X_-(x_0, r, \gamma'(t_0), s)$ and $\gamma([t_0, t_\textrm{out}]) \subseteq X_+(x_0, r, \gamma'(t_0), s)$;
\item $\gamma(t_\textrm{in}), \gamma(t_\textrm{out}) \in \partial B_r(x_0)$;
\item $\gamma(s) \in B_r(x_0)$ for every $s \in (t_\textrm{in},t_\textrm{out})$.
\end{enumerate}
We say that $t_\textrm{in}$ and $t_\textrm{out}$ are entrance and exit times of $\gamma$ inside the cone.
\end{dfn}
Proper crossing holds around regular points of any Lipschitz curve, as stated below.
\begin{lem} \label{lem:derivative}
Let $\gamma \in \lipone$ be a curve of finite length and $x_0$ a regular point in $\img \gamma$. Take a preimage $t_0 \in \gamma^{-1}(x_0)$ and $s\in (0,1)$. Then
\begin{equation} \label{eq:radius}
r_0 := \sup \set{r \ge 0 : \text{$\gamma$ crosses the cone $X(\gamma(t_0), r, \spn \gamma'(t_0), s)$ properly at $t_0$}}
\end{equation}
is nonzero and for all $r \in (0, r_0)$, $\gamma$ crosses 
 $X(\gamma(t_0), r, \spn \gamma'(t_0), s)$ properly at $t_0$. 
\end{lem}
\begin{proof}
Since $\gamma$ is differentiable at $t_0$, as $\delta \to 0$ we have:
\begin{equation*}d(\gamma\left(t_0+\delta) - x_0, \spn \gamma'(t_0)\right) \le |\gamma(t_0+\delta) - x_0 - \delta\gamma'(t_0)| = o(\delta),
\end{equation*}
Hence for every $s \in (0,1)$ there exists $\delta_0>0$ such that
\begin{equation*}
d(\gamma\left(t_0 + \delta) - x_0, \spn \gamma'(t_0)\right) \leq s\abs{\gamma(t_0+\delta)-x_0},
\end{equation*}
whenever $|\delta|\leq\delta_0$.
Moreover, there exists $0 < \delta_1 \leq \delta_0$ such that for any $\delta \in [0, \delta_1]$, we have $\pm (\gamma(t_0\pm\delta)-x_0)\cdot \gamma'(t_0) \geq 0$. Hence, we obtain
\begin{equation}\label{eq:crosscone1}
\gamma([t_0-\delta_1, t_0]) \subseteq X_-(x_0, \infty, \gamma'(t), s) \text{ and } \gamma([t_0, t_0+\delta_1]) \subseteq X_+(x_0, \infty, \gamma'(t), s).
\end{equation}
Denote $r_0 = \min \{\abs{\gamma(t_0-\delta_1)-x_0},\abs{\gamma(t_0+\delta_1)-x_0}\}$. By continuity of $\gamma$ it follows
\begin{equation}\label{eq:crosscone2}
\gamma([t_0-\delta_1,t_0]) \cap \partial B_{r_0}(x_0) \neq \emptyset \text{ and } \gamma([t_0, t_0 +\delta_1]) \cap \partial B_{r_0}(x_0) \neq \emptyset.
\end{equation}
From \cref{eq:crosscone1,eq:crosscone2}, we obtain that $\gamma$ crosses the cone $X(x_0, r_0, \spn \gamma'(t_0), s)$ properly. It is clear that the same holds taking $r\leq r_0$.
\end{proof}

\subsection{Slicing traffic plans}

In this section we introduce a new tool, which is the Lagrangian counterpart to the slicing of currents. We refer to \cite{Sim} for a complete presentation of the latter. We begin by defining a localized version of the $\alpha$-energy. For any $\alpha \in [0, 1]$ and any Borel set $E \subseteq \R^d$, we set:
\begin{equation*}
\ener^\alpha({\tplan P}, E) := \int_\lipone \int_{\R_+} \theta_{\tplan P}^{\alpha - 1}(\gamma(t)) \one_{\gamma(t) \in E} |\gamma'(t)| \dd{t} \dd \tplan P(\gamma).
\end{equation*}
By the Area Formula and Fubini's Theorem, if ${\tplan P}$ is rectifiable then it can be expressed as:
\begin{equation*}
\ener^\alpha({\tplan P}, E) = \int_E \theta_{\tplan P}^{\alpha - 1} \Theta_{\tplan P} \dd{\haus^1}.
\end{equation*} 

\begin{prp} \label{prop:slice0}
Let $f : \R^d \to \R$ be a Lipschitz map, let ${\tplan P}$ be a traffic plan, and let $a < b$ be real numbers. Then:
\begin{equation}\label{e:slice_intensity}
\int_a^b \int_\lipone \haus^0((f \circ \gamma)^{-1}(\ell)) \dd \tplan P(\gamma) \dd{\ell} \le \lipconst (f) \, \ener^1({\tplan P}, f^{-1}([a, b])).
\end{equation}
\end{prp}

\begin{proof}
Let us denote $g(\gamma,\ell):=\haus^0((f \circ \gamma)^{-1}(\ell))$ for every $(\gamma, \ell) \in \lipone\times [a,b]$.
By  Fubini's theorem and the Area Formula we compute 
\begin{align*}
\int_a^b \int_\lipone g(\gamma,\ell) \dd \tplan P(\gamma) \dd{\ell} &= \int_\lipone \int_a^b g(\gamma,\ell) \dd{\ell} \dd \tplan P(\gamma) \\
&= \int_\lipone \int_{(f \circ \gamma)^{-1}([a, b])} | (f \circ \gamma)'(t)| \dd{t} \dd \tplan P(\gamma) \\
&\le \lipconst (f) \int_\lipone \int_{\R_+} \one_{\gamma(t) \in f^{-1}([a, b])} |\gamma'(t)| \dd{t} \dd \tplan P(\gamma) \\
&= \lipconst(f)\, \ener^1({\tplan P}, f^{-1}([a, b])).\qedhere
\end{align*}
\end{proof}

\Cref{prop:slice0} allows to give the following definition:

\begin{dfn}[Slice and intensity of a slice]
Let $f : \R^d \to \R$ be a Lipschitz map, and let ${\tplan P}$ be a traffic plan. By \Cref{prop:slice0} the integrand in the left hand side of \Cref{e:slice_intensity} is finite for a.e. $\ell\in\R$. For such values of $\ell$ we denote by $\islice{\tplan P, f,\ell}$ the finite positive measure
$$\islice{\tplan P, f,\ell}  := \int_\lipone \sum_{t \in (f \circ \gamma)^{-1}(\ell)} \delta_{\gamma(t)} \dd \tplan P(\gamma),$$
where the integration is in the sense of \cref{int:meas}, and we call such measure the \emph{slice intensity} of $\tplan P$ with respect to $f$ at level $\ell$. When the slice intensity is defined, we denote by $\slice{\tplan P, f,\ell}$  the \emph{slice} of ${\tplan P}$ with respect to $f$ at level $\ell$, namely the real-valued measure
$$
\slice{\tplan P, f,\ell} := \int_\lipone \sum_{t \in (f \circ \gamma)^{-1}(\ell)} \sign \left((f \circ \gamma)'(t)\right) \delta_{\gamma(t)} \dd \tplan P(\gamma).
$$
\end{dfn}


\begin{prp} \label{prop:slice}
Let ${\tplan P}$ be a rectifiable traffic plan, and let $a < b$ be real numbers. Then:
$$\int_a^b \mass^\alpha(\islice{\tplan P, f,\ell}) \dd{\ell} \le \lipconst(f) \, \ener^\alpha({\tplan P}, f^{-1}([a, b])).$$
\end{prp}

\begin{proof}
The network $\Sigma_{\tplan P}$ is a $1$-rectifiable set. So we know (see \cite[Chapter~3, Remark~2.10]{Sim}) that for almost every $\ell$, $\Sigma_{\tplan P} \cap \{ f = \ell \}$ is a $0$-rectifiable set, that is to say it is at most countable. In addition, it is true that for almost every $\ell$:
\begin{equation} \label{eq:sig}
\textrm{for ${\tplan P}$-almost every $\gamma \in \lipone$, $\haus^0(\img \gamma \cap \{f = \ell\} \setminus \Sigma_{\tplan P}) = 0$,}
\end{equation}
i.e. $\img \gamma \cap \{f = \ell\} \subseteq \Sigma_{\tplan P}$. Indeed, by Fubini's Theorem and the Area Formula:
\begin{multline*}
\int_\R \int_\lipone \haus^0(\img \gamma \cap \{f = \ell\} \setminus \Sigma_{\tplan P}) \dd \tplan P(\gamma) \dd{\ell}\\
\begin{aligned}
&= \int_\lipone \int_\R \int_{(f \circ \gamma)^{-1}(\ell)} \one_{\gamma(t) \not\in \Sigma_{\tplan P}} \dd{\haus^0}(t) \dd{\ell} \dd \tplan P(\gamma) \\
&\leq \int_\lipone \int_{\R_+} |(f \circ \gamma)'(t)| \one_{\gamma(t) \not\in \Sigma_{\tplan P}} \dd{t} \dd \tplan P(\gamma) \\
&\leq \lipconst (f) \int_\lipone \int_{\img \gamma \setminus \Sigma_{\tplan P}} \haus^0\left(\gamma^{-1}(x)\right) \dd{\haus^1}(x) \dd \tplan P(\gamma),
\end{aligned}
\end{multline*}
which equals $0$ since ${\tplan P}$ is assumed to be rectifiable.

Now let $\varphi \in C_c^0(\R^d)$ be nonnegative. Let $\ell$ be such that \cref{eq:sig} is true. Then
\begin{align*}
\langle \islice{\tplan P, f,\ell}, \varphi \rangle &= \int_\lipone \sum_{t \in (f \circ \gamma)^{-1}(\ell)} \varphi(\gamma(t)) \dd \tplan P(\gamma) \\
&= \int_\lipone \sum_{x \in f^{-1}(\ell) \cap \Sigma_{\tplan P}} \haus^0\left(\gamma^{-1}(x)\right) \varphi(x) \dd \tplan P(\gamma) \\
&= \sum_{x \in f^{-1}(\ell) \cap \Sigma_{\tplan P}} \varphi(x) \int_\lipone \haus^0\left(\gamma^{-1}(x)\right) \dd \tplan P(\gamma) ,
\end{align*}
using Fubini's Theorem for the last equality. Thus
\[\islice{\tplan P, f,\ell} = \Theta_{\tplan P} \haus^0 \mres (f^{-1}(\ell) \cap \Sigma_{\tplan P}),\]
and $\mass^\alpha(\islice{\tplan P, f,\ell}) = \sum_{x \in f^{-1}(\ell) \cap \Sigma_{\tplan P}} \Theta_{\tplan P}(x)^\alpha$.
Finally, by Fubini's Theorem and the Area Formula again, we compute
\begin{align*}
\int_a^b \mass^\alpha(\islice{\tplan P, f,\ell}) \dd\ell &= \int_a^b \sum_{x \in f^{-1}(\ell) \cap \Sigma_{\tplan P}} \Theta_{\tplan P}(x)^{\alpha - 1} \Theta_{\tplan P}(x) \dd{\ell} \\
&= \int_a^b \sum_{x \in f^{-1}(\ell) \cap \Sigma_{\tplan P}} \Theta_{\tplan P}(x)^{\alpha - 1} \int_\lipone \haus^0\left(\gamma^{-1}(x)\right) \dd \tplan P(\gamma) \dd{\ell} \\
&= \int_\lipone \int_a^b \sum_{t \in (f \circ \gamma)^{-1}(\ell)} \one_{\gamma(t) \in \Sigma_{\tplan P}} \Theta_{\tplan P}\left(\gamma(t)\right)^{\alpha - 1} \dd{\ell} \dd \tplan P(\gamma) \\
&= \int_\lipone \int_{\R_+} \one_{\gamma(t) \in \Sigma_{\tplan P}, (f \circ \gamma)(t) \in [a, b]} \Theta_{\tplan P}\left(\gamma(t)\right)^{\alpha - 1} |(f \circ \gamma)'(t)| \dd{t} \dd \tplan P(\gamma) \\
&\le \lipconst (f) \, \ener^\alpha({\tplan P}, f^{-1}([a, b])).\qedhere
\end{align*}
\end{proof}

\section{From cancellations to cycles}\label{sec:canceltocycles}

Take $\{\tplan P_n\}_{n\in \N}$ and $\tplan P$ satisfying the hypotheses of \Cref{mainthm} and set $T = T_{\tplan P}$. At first glance, ${\tplan P}$ could fail to be a good decomposition of $T$: in general a limit of good decompositions is not a good decomposition. In order to show that ${\tplan P}$ is in fact a good decomposition of $T$, we are going to prove that, as a limit of optimal traffic plans, it cannot produce cancellations at the Eulerian level. In the following we will always assume that $\tplan P_n$ and $\tplan P$ are rectifiable traffic plans, which is not restrictive in view of Theorem 4.10 of \cite{BCM}. 

\subsection{Cancellations}\label{sec:cancel}

Cancellations in $\tplan P$ mean \enquote{pieces of trajectories} that disappear in the induced current, due to positive amounts of curves going in opposite directions. Let us be more precise.

In general the density of the induced current $\vec \theta_{\tplan P}$ is less or equal than the full multiplicity $\Theta_{\tplan P}$, $\haus^1$-a.e. on $\Sigma_{\tplan P}$. Indeed, take a curve $\gamma \in \lipone$ such that $\haus^1(\img\gamma \setminus \Sigma_{\tplan P}) = 0$, and a regular point $x \in\img\gamma$ where $\tanspace(x,\Sigma_{\tplan P})$ exists. Note that by the triangle inequality:
\begin{equation} \label{eq:ti}
\abs{\vec m_\gamma(x)} \le \# \gamma^{-1}(x),
\end{equation}
hence taking a regular point $x$ of $\tplan P$, one has:
\begin{align}
\abs{\vec \theta_{\tplan P}(x)} = \abs*{\int_{\lipone} \vec m_\gamma(x) \dd\tplan P(\gamma)} &\leq \int_{\lipone} \abs{\vec m_\gamma(x)} \dd\tplan P(\gamma)\label{eq:strict1}\\
&\leq \int_{\lipone} \# \gamma^{-1}(x) \dd\tplan P(\gamma) = \Theta_{\tplan P}(x).\label{eq:strict2}
\end{align}
\begin{rmk}\label{ineq_good_decompo}
These inequalities and their equality case are closely related to the notion of good decomposition. Indeed, notice that using Fubini's Theorem, \cref{eq:strict1} is an equality $\haus^1$-a.e. if and only if
\begin{align*}
\mass(T) = \int_\lipone \int_{\Sigma_{\tplan P}} \abs{\vec m_\gamma(x)} \dd\haus^1(x) \dd\tplan P(\gamma) = \int_\lipone \mass(I_\gamma) \dd\tplan P(\gamma),
\end{align*}
that is if \cref{prop:B} of \Cref{gooddecompo} holds. Moreover \cref{prop:A} implies equality $\haus^1$-a.e. in \cref{eq:strict2} as well as $\theta_{\tplan P}(x) = \Theta_{\tplan P}(x)$.
\end{rmk}
We say that $\tplan P$ has cancellations if we have a strict inequality:
\begin{equation}\label{eq:cancel}
\abs{\vec \theta_{\tplan P}(x)} < \Theta_{\tplan P}(x),
\end{equation}
on a subset of $\Sigma_{\tplan P}$ of positive $\haus^1$-measure. Notice that \cref{eq:cancel} happens if either inequality \cref{eq:strict1} or \cref{eq:strict2} is strict. Inequality \cref{eq:strict2} is strict if there exists a positive amount of curves $\gamma$ such that equality \cref{eq:ti} is strict, and since $\gamma'(t)$ belongs to $\tanspace(x,\Sigma_{\tplan P})$ for all $t\in \gamma^{-1}(x)$ by \Cref{aeregul}, it means that $\gamma$ crosses $x$ at least twice in opposite directions: heuristically these cancellations are due to those particles flowing through the same point at least twice, with different orientations. Inequality \cref{eq:strict1} is strict if there are two sets of curves $A_x, B_x \in \lipone$ of positive measure such that the resultant tangent $\vec m_\gamma(x)$ belongs to $\tanspace(x,\Sigma_{\tplan P})$ with a certain orientation on $A_x$ and with the opposite orientation on $B_x$: heuristically these cancellations are due to the interactions between different particles flowing in opposite directions.

To capture both situations at the same time, we choose once and for all a (Borel measurable) orientation $\tau_{\Sigma_{\tplan P}}(x)$ for $\tanspace(x,\Sigma_{\tplan P})$, and introduce for every $x \in \R^d$ the sets:
\begin{gather}
\Gamma^\pm(x) := \set{\gamma \in \lipone : \exists t \in (0,T(\gamma)) \text{ s.t. } \gamma'(t)/|\gamma'(t)| = \pm \tau_{\Sigma_{\tplan P}}(x)},\\
\shortintertext{as well as the corresponding multiplicities}
\theta_{\tplan P}^\pm(x) := \tplan P \left( \Gamma^\pm(x) \right),
\end{gather}
\begin{equation}\label{eq:bartheta}
\bar\theta_{\tplan P}(x) := \min\{\theta_{\tplan P}^+(x), \theta_{\tplan P}^-(x)\}.
\end{equation}
We may characterize cancellations at $x$ as stated in the following lemma:
\begin{lem}\label{cancelcharac}
Take a regular point $x$ of $\tplan P$. The following assertions are equivalent:
\begin{enumerate}
\item $\abs{\vec \theta_{\tplan P}(x)} = \Theta_{\tplan P}(x)$,
\item $\bar\theta_{\tplan P}(x) = 0$,
\item there exists $s\in \{-1,+1\}$ such that for $\tplan P$-a.e. curve and all $t \in \gamma^{-1}(x)$,
\[\gamma'(t) = s \abs{\gamma'(t)} \tau_{\Sigma_{\tplan P}}(x).\]
\end{enumerate}
\end{lem}

\subsection{Existence of Lagrangian cycles}

From the perspective of reasoning by contradiction, the goal of this section is to study properties of traffic plans producing cancellations. The theorem below guarantees the existence of \enquote{Lagrangian cycles} in ${\tplan P}$, that is to say two families of curves with positive measures passing through two distinct points $x$ and $y$ in opposite order, namely ${\tplan P}(\Gamma(x, y)), {\tplan P}(\Gamma(y, x)) > 0$, where for any $u, v \in \R^d$
\begin{equation*}
\Gamma(u, v) := \set{\gamma \in \lipone : \exists s \le t: \gamma(s) = u, \gamma(t) = v}.
\end{equation*}
These cycles are obviously obstacles to $\tplan P$ being optimal, but at this point it is not yet a contradiction, as the optimality of $\tplan P$ is precisely what we want to prove.

\begin{thm}[Existence of Lagrangian cycles] \label{thm:cycles}
Let ${\tplan P}$ be a traffic plan with finite energy, and assume $\haus^1(\{ \bar \theta_{\tplan P} > 0\}) > 0$. Then there exists $F \subseteq \{ \bar \theta_{\tplan P} > 0\}$ with positive $\haus^1$-measure such that for every $x_0 \in F$, there exists $G \subseteq F$ with positive $\haus^1$-measure satisfying:
\begin{equation*}
\forall x \in G, \quad \min \left\{ {\tplan P}(\Gamma(x_0, x)), {\tplan P} (\Gamma(x, x_0)) \right\} \ge \bar \theta_{\tplan P}(x_0) /4.
\end{equation*}
\end{thm}

Before proving this theorem, we need the following lemma, which describes the geometric situation on small balls $B_r(x_0)$ around every point $x_0$ in a suitable subset $F$ of $\{\bar\theta_{\tplan P} > 0\}$: \enquote{most} of $\Sigma_{\tplan P}$ lies inside $F$, itself contained in a cone which is crossed properly, and in both directions, by a fixed amount of curves. \begin{lem} \label{lem:properties}
Under the assumptions of \Cref{thm:cycles}, there exists $F \subseteq \{ \bar \theta_{\tplan P} > 0\}$ with positive $\haus^1$-measure such that $\haus^1$-almost every $x_0 \in F$ satisfies:
\begin{enumerate}[(i)]
\item\label{prop:i} $\displaystyle \haus^1\left(F \cap B_r(x_0)\right) \sim 2r$ as $r \to 0$;
\item\label{prop:ii} $\displaystyle \int_{(\Sigma_{\tplan P} \setminus F) \cap B_r(x_0)} \theta_{\tplan P} \dd{\haus^1} = o(r)$;
\item\label{prop:iii} for all $s \in (0, 1)$, there exists $r_1 > 0$ such that $F \cap B_{r_1}(x_0)$ is contained in the cone $X(x_0, r_1, \spn \tau_{\Sigma_{\tplan P}}(x_0), s)$;
\item\label{prop:iv} for all $s \in (0, 1)$, there exists $r_0 > 0$ and two (not necessarily disjoint) Borel sets of curves $\Lambda_{r_0}^\pm \subseteq \Gamma^\pm(x_0)$, as well as Borel maps $t_0^\pm : \Lambda_{r_0}^\pm \to \R_+$ such that:
\begin{itemize}
\item $\min \{{\tplan P}(\Lambda_{r_0}^+), {\tplan P}(\Lambda_{r_0}^-) \} \ge \bar \theta_{\tplan P} (x_0)/2$;
\item every $\gamma \in \Lambda_{r_0}^\pm$ crosses the cone $X(x_0, r_0, \spn \tau_{\Sigma_{\tplan P}}(x_0), s)$ properly at time $t_0^\pm(\gamma)$;
\item for any $0 < r \le r_0$ there exist Borel maps $t_{r, in}^\pm$, $t_{r, out}^\pm$ from $\Lambda_{r_0}^\pm$ to $\R_+$ which are entrance and exit times in the cone $X(x_0, r, \spn \tau_{\Sigma_{\tplan P}}(x_0), s)$ for every curve in $\Lambda_{r_0}^\pm$.
\end{itemize}
\end{enumerate}
\end{lem}

\begin{proof}

Since $\Sigma_{\tplan P}$ is 1-rectifiable, it is contained in the union of an $\haus^1$-negligible set and of countably many images of Lipschitz curves of finite length. Thus, there exists a Lipschitz curve $\bar \gamma$ such that $\haus^1(\{ \bar \theta_{\tplan P} > 0\} \cap \img \bar \gamma) > 0$. We set $F := \{ \bar \theta_{\tplan P} > 0\} \cap \img \bar \gamma$. We want to show each item holds for $\haus^1$-almost every $x_0 \in F$.

\begin{innerproof}[Proof of \cref{prop:i}]
We know that $\img \bar \gamma$ is 1-rectifiable and $\haus^1(\img \bar \gamma) < \infty$. Then \cite[Theorem 17.6]{Mat} implies that $\haus^1(\img \bar \gamma \cap B_r(x_0)) \sim 2r$ as $r \to 0$ for $\haus^1$-almost every $x_0 \in \img \bar \gamma$. Moreover, almost every $x_0 \in F \subseteq \img \bar \gamma$ is a density point of the function $\one_F$ with respect to the Radon measure $\haus^1 \mres \img \bar \gamma$ (see \cite[Theorem 1.32]{EG}) i.e. $\haus^1\left( F \cap B_r(x_0) \right) \sim \haus^1\left( \img \bar \gamma \cap B_r(x_0) \right)$, hence the result.
\end{innerproof}

\begin{innerproof}[Proof of \cref{prop:ii}]
By the same argument, almost every $x_0 \in F \subseteq \Sigma_{\tplan P}$ is a density point of the function $\one_F$ with respect to the Radon measure $\theta_{\tplan P} \haus^1 \mres \Sigma_{\tplan P}$ so that
\begin{equation} \label{eq:o(int)}
\int_{(\Sigma_{\tplan P} \setminus F) \cap B_r(x_0)} \theta_{\tplan P} \dd{\haus^1} = o \left( \int_{\Sigma_{\tplan P} \cap B_r(x_0)} \theta_{\tplan P} \dd{\haus^1} \right).
\end{equation}
Yet a subset of $F \subseteq \Sigma_{\tplan P}$ which is negligible for $\theta_{\tplan P} \haus^1 \mres \Sigma_{\tplan P}$ is also $\haus^1$-negligible, so it is still true for $\haus^1$-almost all $x_0 \in F$.

In addition, for $\haus^1$-almost every $x_0 \in \Sigma_{\tplan P}$, there exist $c = c(x_0) > 0$ and $\rho = \rho(x_0) > 0$ such that
\begin{equation} \label{eq:O(r)}
\forall r \le \rho, \quad \int_{\Sigma_{\tplan P} \cap B_r(x_0)} \theta_{\tplan P} \dd{\haus^1} \le c r.
\end{equation}
Indeed, let us show by contraposition that the set
\[A := \set*{x \in \Sigma_{\tplan P} : \limsup_{r \to 0} \frac{1}{r} \int_{\Sigma_{\tplan P} \cap B_r(x_0)} \theta_{\tplan P} \dd{\haus^1} = + \infty}\]
is $\haus^1$-negligible. Let $c,\varepsilon > 0$. For every $x \in A$, there exists $r(x) \in (0, \varepsilon]$ satisfying
\begin{equation} \label{eq:forvitali}
\frac{1}{r(x)} \int_{\Sigma_{\tplan P} \cap \bar{B}_{r(x)}(x)} \theta_{\tplan P} \dd{\haus^1} \ge c.
\end{equation}

The family $\set{\bar{B}_{r(x)}(x)}_{x \in A}$ is a covering of $A$ and for every $x \in A$, $r(x) \le \varepsilon$. Then by Vitali's Covering Theorem (\cite[Theorem 2.1]{Mat}), one may extract a (finite or countable) sequence $\{B_i\}_{i \in I} \subseteq \set{\bar{B}_{r(x)}(x) : x \in A}$ of disjoint closed balls such that $A \subseteq \bigcup_{i \in I} \hat{B}_i$, where $\hat{B}_i$ is the concentric ball to $B_i$ with radius 5 times the radius of $B_i$. Therefore we get the following inequalities:
\[
\haus_{5 \varepsilon}^1(A) \le \sum_{i \in I} \diam \hat{B}_i \le 5 \sum_{i \in I} \diam B_i \refabove{\cref{eq:forvitali}}{\le} \frac{10}{c} \sum_{i \in I} \int_{\Sigma_{\tplan P} \cap B_i} \theta_{\tplan P} \dd{\haus^1} \le \frac{10}{c} \int_{\Sigma_{\tplan P}} \theta_{\tplan P} \dd{\haus^1}
\]
where the last inequality is due to the fact that the balls $B_i$ are disjoint. Since 
\begin{equation*}
\begin{split}
\int_{\Sigma_{\tplan P}} \theta_{\tplan P} \dd{\haus^1}&= \int_{\Sigma_{\tplan P}} \left ( \frac{\theta_{\tplan P}}{\mathbb M(\tplan P)}\right ) \mathbb M(\tplan P) \dd{\haus^1} \\
&\le \int_{\Sigma_{\tplan P}} \left ( \frac{\theta_{\tplan P}}{\mathbb M(\tplan P)}\right )^\alpha \mathbb M(\tplan P) \dd{\haus^1} = \int_{\Sigma_{\tplan P}} \theta_{\tplan P}^\alpha \mathbb M(\tplan P)^{1-\alpha} \dd{\haus^1}< \infty
\end{split}
\end{equation*}
and $c$ is arbitrary, we find that for all $\varepsilon > 0$, $\haus_{5 \varepsilon}^1(A) = 0$, which yields $\haus^1(A) = 0$.

Therefore, \cref{eq:o(int)} and \cref{eq:O(r)} hold for $\haus^1$-almost every point in $F$, hence the result.
\end{innerproof}

\begin{innerproof}[Proof of \cref{prop:iii}]
Recall that there is a Lipschitz curve $\bar \gamma \subseteq \Sigma_{\tplan P}$ whose image contains $F$. By \Cref{aeregul}, at $\haus^1$-almost every point $x \in \img \bar \gamma$, $\bar{\gamma}^{-1}(x)$ is finite and for all $t \in \bar{\gamma}^{-1}(x)$, $\spn \bar{\gamma}'(t) = \spn \tau_{\Sigma_{\tplan P}}(x)$. Let $x_0 \in \img \bar \gamma$ be such a point. For every $t \in \bar{\gamma}^{-1}(x_0)$, we apply \Cref{lem:derivative} to get a radius $r_t > 0$ such that $\bar \gamma$ lies in the cone $X(x_0, r_t, \spn \tau_{\Sigma_{\tplan P}}(x_0), s)$ in a small interval $I_t = ]t-\delta_t,t+\delta_t[$. Set $r = \min \set{r_t > 0 : \bar\gamma(t) = x_0} > 0$, then taking $r_1 \leq r$ small enough to make sure that $B_{r_1}(x_0) \cap \gamma\left(\R_+\setminus \bigcup_t I_t\right) = \emptyset$ leads to the desired conclusion.
\end{innerproof}

\begin{innerproof}[Proof of \cref{prop:iv}]
Let $x_0 \in F$ be a regular point for $\tplan P$ and set $V := \spn \tau_{\Sigma_{\tplan P}}(x_0)$. The function
\begin{align*}
t_0^\pm : \Gamma^\pm(x_0) &\to \R_+ \\
\gamma &\mapsto \inf \set{t \in \R_+ : \gamma(t) = x_0, \gamma'(t)/|\gamma'(t)| = \pm \tau_\Sigma(x_0)}
\end{align*}
is well-defined and $t_0^\pm(\gamma) \in (0, T_\infty(\gamma))$ for $\tplan P$-a.e. $\gamma \in \Gamma^\pm(x_0)$, as $x_0$ is a regular point for $\tplan P$. Then fix $s \in (0, 1)$. We denote by $r_0^\pm(\gamma)$ the (positive) radius given by \cref{eq:radius} in \Cref{lem:derivative} and we set $\Lambda_r^\pm := \set{\gamma \in \Gamma^\pm(x_0) : r_0^\pm(\gamma) > r}$ for any $r > 0$. Since $r_0^\pm(\gamma) > 0$ for every $\gamma \in \Gamma^\pm(x_0)$ and $\{\Lambda_{1/n}^\pm\}_{n\in\N^*}$ are a nested family of set such that $\bigcup_{n \in \N^\star} \Lambda_{1/n}^\pm = \Gamma^\pm(x_0)$, then there exists $n \in \N$ such that ${\tplan P}(\Lambda_{1/n}^+) \ge {\tplan P}(\Gamma^+(x_0))/2 \ge \bar \theta_{\tplan P} (x_0)/2$ and ${\tplan P}(\Lambda_{1/n}^-) \ge {\tplan P}(\Gamma^-(x_0))/2 \ge \bar \theta_{\tplan P} (x_0)/2$. Therefore, setting $r_0 := 1/n$, every $\gamma \in \Lambda_{r_0}^\pm$ crosses the cone $X(x_0, r_0, V, s)$ properly at time $t_0^\pm(\gamma)$. This is also true for all the homothetic cones with radius $0 < r \le r_0$ (see \Cref{lem:derivative}). Thus for every  $\gamma \in \Lambda_{r_0}^\pm$ we define the entrance and exit times as:
\begin{align*}
t_{r, in}^\pm(\gamma) : \Lambda_{r_0}^\pm &\to \R_+\\
\gamma &\mapsto \sup \set{t \le t_0^\pm(\gamma) : \gamma(t) \not\in B_r(x_0)},\\
t_{r, out}^\pm(\gamma) : \Lambda_{r_0}^\pm &\to \R_+ \\
\gamma &\mapsto \inf \set{t \ge t_0^\pm(\gamma) : \gamma(t) \not\in B_r(x_0)}. \qedhere
\end{align*}
\end{innerproof}
\end{proof}

We can now go back to the existence of Lagrangian cycles in ${\tplan P}$:
\begin{proof}[Proof of \Cref{thm:cycles}]
Let $F \subseteq \{ \bar \theta_{\tplan P} > 0 \}$ given by \Cref{lem:properties} and let $x_0 \in F$ satisfying \cref{prop:i,prop:ii,prop:iii,prop:iv}. We fix $s \in (0, 1)$ (for example $s = 1/2$) and we take $r_0 > 0$, $\Lambda_{r_0}^\pm$ and $t_{r, in}^\pm$, $t_{r, out}^\pm$ as in \cref{prop:iii,prop:iv}. For any $0 < r \le r_0$, we define $\tplan Q_r^\pm := (g_r)_\sharp ({\tplan P} \mres \Lambda_{r_0}^\pm)$ where
\begin{align*}
g_r : \Lambda_{r_0}^\pm &\to \lipone \\
\gamma &\mapsto \gamma_{\vert [t_{r, in}^\pm(\gamma), t_{r, out}^\pm(\gamma)]}.
\end{align*}
This is obviously a traffic plan. Let us estimate its multiplicity at $x \in \R^d$:
\begin{equation*}
\theta_{\tplan Q_r^\pm}(x) = \int_{\Lambda_{r_0}^\pm} \one_{x \in \img g_r(\gamma)} \dd \tplan P(\gamma) \le \int_\lipone \one_{x \in \img \gamma} \dd \tplan P(\gamma) = \theta_{\tplan P}(x),
\end{equation*}
so that
\begin{equation} \label{eq:rest}
\int_{(\Sigma_{\tplan P} \setminus F) \cap B_r(x_0)} \theta_{\tplan Q_r^\pm} \dd{\haus^1} \le \int_{(\Sigma_{\tplan P} \setminus F) \cap B_r(x_0)} \theta_{\tplan P} \dd{\haus^1}.
\end{equation}
Furthermore, Fubini's Theorem yields:
\begin{equation} \label{eq:mainpart}
\int_{\Sigma_{\tplan P} \cap B_r(x_0)} \theta_{\tplan Q_r^\pm} \dd{\haus^1} = \int_{\Lambda_{r_0}^\pm} \haus^1(\img g_r(\gamma)) \dd \tplan P(\gamma) \ge 2r {\tplan P}(\Lambda_{r_0}^\pm),
\end{equation}
where the inequality comes from the fact that every curve in $\Lambda_{r_0}^\pm$ crosses the cone $X(x_0, r, V, s)$ properly, hence their length between the entrance and exit times is at least $2r$.
Recalling \cref{prop:i,prop:ii} of \Cref{lem:properties}, as well as the fact that $\theta_{\tplan Q_r^\pm} \le {\tplan P}(\Lambda_{r_0}^\pm)$, \cref{eq:mainpart} and \cref{eq:rest} lead to
\begin{align*}
0 &\le \int_{F \cap B_r(x_0)} \left( {\tplan P}(\Lambda_{r_0}^\pm) - \theta_{\tplan Q_r^\pm} \right) \dd{\haus^1}\\
&= \haus^1(F \cap B_r(x_0)){\tplan P}(\Lambda_{r_0}^\pm) - \int_{\Sigma_{\tplan P} \cap B_r(x_0)} \theta_{\tplan Q_r^\pm} \dd{\haus^1} + \int_{(\Sigma_{\tplan P} \setminus F) \cap B_r(x_0)} \theta_{\tplan Q_r^\pm} \dd{\haus^1} \\
&\le \left(\haus^1(F \cap B_r(x_0)) - 2r \right) {\tplan P}(\Lambda_{r_0}^\pm) + \int_{(\Sigma_{\tplan P} \setminus F) \cap B_r(x_0)} \theta_{\tplan P} \dd{\haus^1} \\
&= o(r).
\end{align*}
Thus for any $c \in (0, 1)$, Markov inequality yields
\begin{equation*}
\haus^1\left( \left\{ {\tplan P}(\Lambda_{r_0}^\pm) - \theta_{\tplan Q_r^\pm} > c {\tplan P}(\Lambda_{r_0}^\pm) \right\} \cap F \cap B_r(x_0) \right) = o(r),
\end{equation*}
and therefore
\begin{equation*}
\frac{\haus^1\left( \left\{ \theta_{\tplan Q_r^\pm} \ge (1 - c) {\tplan P}(\Lambda_{r_0}^\pm) \right\} \cap F \cap B_r(x_0) \right)}{\haus^1(F \cap B_r(x_0))} \xstrongto{r\to 0} 1.
\end{equation*}
Now we take $c = 1/2$ and $r > 0$ small enough so that the quotient above (for $\pm$ being $+$ and $-$) is greater than $3/4$. We set $G := \left\{ \theta_{\tplan Q_r^+} \ge {\tplan P}(\Lambda_{r_0}^+)/2 \right\} \cap \left\{ \theta_{\tplan Q_r^-} \ge {\tplan P}(\Lambda_{r_0}^-)/2 \right\} \cap F \cap B_r(x_0)$. As expected, we have $\haus^1(G \cap B_r(x_0)) \ge \haus^1(F \cap B_r(x_0))/2$. Note that by \cref{prop:iii}, $G \subseteq X(x_0, r, \spn \tau_{\Sigma_{\tplan P}}(x_0), s)$. Finally, let $x \in G$ be distinct from $x_0$ and assume for example that $x \in X_+(x_0, r, \tau_\Sigma(x_0), s)$. Since by \cref{prop:iv} every curve in $\Lambda_{r_0}^\pm$ crosses the cone $X(x_0, r, \spn \tau_{\Sigma_{\tplan P}}(x_0), s)$ properly, then every curve $\gamma$ in $\Lambda_{r_0}^+$ such that $x \in \img g_r(\gamma)$ goes through $x_0$ before going through $x$ along the piece $g_r(\gamma)$, and vice versa for the curves in $\Lambda_{r_0}^-$, which yields:
\begin{align*}
\frac{{\tplan P}(\Lambda_{r_0}^+)}{2} \le \theta_{\tplan Q_r^+}(x_0) &= {\tplan P}\left(\set{\gamma \in \Lambda_{r_0}^+ : x \in \img g_r(\gamma)}\right) \le {\tplan P}\left(\Gamma(x_0, x)\right),\\
\frac{{\tplan P}(\Lambda_{r_0}^-)}{2} \le \theta_{\tplan Q_r^-}(x_0) &= {\tplan P}\left(\set{\gamma \in \Lambda_{r_0}^- : x \in \img g_r(\gamma)}\right) \le {\tplan P}\left(\Gamma(x, x_0)\right).
\end{align*}
Analogously, if $x \in X_-(x_0, r, \tau_\Sigma(x_0), s)$, then ${\tplan P}(\Lambda_{r_0}^+)/2 \le {\tplan P}(\Gamma(x, x_0))$ and ${\tplan P}(\Lambda_{r_0}^-)/2 \le {\tplan P}(\Gamma(x_0, x))$.
We conclude recalling ${\tplan P}(\Lambda_{r_0}^\pm) \ge \bar \theta_{\tplan P} (x_0)/2$ from \cref{prop:iv}.
\end{proof}

\section{Cycles and quasi-cycles in traffic plans}\label{sec:cyclesandquasi}

Consider a sequence of traffic plans $\{\tplan P_n\}_{n\in\N}$ with bounded energy which converges to a traffic plan $\tplan P$. Assume there exists $(x,y)\in X\times X$ such that $\tplan P(\Gamma(x,y))>0$ and $\tplan P(\Gamma(y,x))>0$. We show existence of \emph{quasi-cycles} in the $\tplan P_n$'s, namely we prove the following.
Denote for any $u, v \in \R^d$ and $\eps>0$
\begin{equation*}
\Gamma_\varepsilon(u, v) := \set{\gamma \in \lipone : \exists s \le t: \gamma(s) \in B_\varepsilon(u), \gamma(t) \in B_\varepsilon(v)},
\end{equation*}
there exists $\delta>0$ such that for every $\eps >0$ there exists $N \in \N$ such that
\[\min \left\{{\tplan P}_n(\Gamma_\varepsilon(x, y)), {\tplan P}_n(\Gamma_\varepsilon(y, x)) \right\} \geq \delta, \quad \forall n\geq N. \]
These points may be well-chosen to guarantee that the energy of $\tplan P_n$ vanishes somewhat uniformly in $n$ on small balls around them. Then we estimate the energy gain 
obtained by removing such quasi-cycles. To simplify the construction, we will build a competing transport path rather than a traffic plan, but this is not a problem by equivalence of the two frameworks (we can always build a traffic plan with a lower or equal cost).

\subsection{From cycles to quasi-cycles}\label{sec:cyclestoquasi}

We start with the lemma controlling the energy on small balls: for almost every $x$, the energy of $T_n$ on small balls $B_\varepsilon(x)$ becomes arbitrarily small uniformly on a subsequence, as $\varepsilon$ goes to $0$. The lemma is proven for transport paths as justified before.

\begin{lem} \label{lem:concentration}
Let $\{ T_n \}_{n \in \N}$ be a sequence of transport paths such that $\sup_{n \in \N} \mass^\alpha(T_n) < \infty$. Then, one has for $\haus^1$-almost every $x \in \R^d$:
\begin{equation} \label{eq:energyball}
\liminf_{n \to \infty} \mass^\alpha(T_n \mres B_\varepsilon(x)) \xstrongto{\eps\to 0} 0.
\end{equation}
\end{lem}

\begin{proof}
We prove \cref{eq:energyball} by a simple covering argument (the same as in \cref{prop:ii} of \Cref{lem:properties} but in an $\haus^0$ fashion). Set
\begin{equation*}
A_p := \set{x \in \R^d : \forall \varepsilon_0 > 0, \exists \varepsilon \le \varepsilon_0 \text{ s.t. } \liminf_{n \to \infty} \mass^\alpha(T_n \mres B_\varepsilon(x)) \ge 1/p},
\end{equation*}
and take $k$ distinct points $x_i$ in this set, as well as suitable radii $r_i > 0$ so that the balls $\bar{B}_{r_i}(x_i)$ are disjoint and for every $i$, $\liminf_{n \to \infty} \mass^\alpha(T_n \mres B_{r_i}(x_i)) \ge 1/p$. We get
\begin{equation*}
\frac{k}{p} \le \sum_{i = 1}^k \liminf_{n \to \infty} \mass^\alpha(T_n \mres B_{r_i}(x_i)) \le \liminf_{n \to \infty} \mass^\alpha(T_n) \le \sup_{n \in \N} \mass^\alpha(T_n).
\end{equation*}
So $k$ is bounded, hence $\haus^0(A_p) < \infty$. Therefore $\bigcup_{p \in \N^\star} A_p$ is at most countable.
\end{proof}




We continue with the existence of quasi-cycles.

\begin{lem} \label{lem:quasicycles}
Let $\{ {\tplan P}_n \}_{n \in \N}$ and ${\tplan P}$ be traffic plans such that ${\tplan P}_n \weakstarto {\tplan P}$. Assume there exist $x, y \in \R^d$ and $\delta > 0$ satisfying $\min \left\{{\tplan P}(\Gamma(x, y)) , {\tplan P}(\Gamma(y, x))\right\} \ge \delta$. Then:
\begin{equation*}
\forall \varepsilon > 0, \exists N \in \N \quad \textrm{s.t.} \quad \forall n \ge N, \quad \min \left\{{\tplan P}_n(\Gamma_\varepsilon(x, y)), {\tplan P}_n(\Gamma_\varepsilon(y, x))\right\} \ge \delta/2.
\end{equation*}
\end{lem}

\begin{proof}
Notice that $\Gamma_\varepsilon(x, y)$ is an open subset of $\lipone$ (recall that it is endowed with the topology of uniform convergence on compact subsets of $\R_+$). Indeed, take $\gamma \in \Gamma_\varepsilon(x, y)$ and denote $s < t$ such that $\gamma(s) \in B_\varepsilon(x)$ and $\gamma(t) \in B_\varepsilon(x)$. Then any curve $\tilde{\gamma}$ such that $\| \tilde{\gamma} - \gamma \|_{\infty, [0, t]} < \min \{ \varepsilon - |\gamma(s) - x|, \varepsilon - |\gamma(t) - y| \}$ belongs to $\Gamma_\varepsilon(x, y)$. Thus $\liminf_{n \to \infty} {\tplan P}_n(\Gamma_\varepsilon(x, y)) \ge {\tplan P}(\Gamma_\varepsilon(x, y))\ge {\tplan P}(\Gamma(x, y)) \ge \delta$ (by \cite[Theorem 1.40]{EG}), hence the result. The same holds true for $\Gamma_\varepsilon(y, x)$ after exchanging $x$ and $y$.
\end{proof}

\subsection{Removing quasi-cycles}

Here we show that if $\tplan P$ has an \enquote{$\eps$-cycle} of mass $m$ in the sense that
\[\min \left\{{\tplan P}(\Gamma_\varepsilon(x, y)), {\tplan P}(\Gamma_\varepsilon(y, x)) \right\} \geq m,\]
then one may do a shortcut to reduce the  $\alpha$-energy of $\tplan P$ up to error terms equal to the energy of $\tplan P$ on the balls $B_{2\eps}(x), B_{2\eps}(y)$.

\begin{prp} \label{prop:comp}
Let ${\tplan P} \in \tplanset(\mu^-,\mu^+)$ be a traffic plan with finite energy supported on the set of simple curves. Assume that there exists $\eps_0 \in (0, |y - x|/8]$ such that
\[m := \min \{ {\tplan P}(\Gamma_{\varepsilon_0}(x, y)), {\tplan P}(\Gamma_{\varepsilon_0}(y, x)) \} > 0.\]
Then there exists $\bar T \in \tpathset(\mu^-,\mu^+)$ such that
\begin{equation*}
\mass^\alpha(\bar T) \le \ener^\alpha({\tplan P}) - \alpha {\tplan P}(\lipone)^{\alpha - 1} m |y - x| + \ener^\alpha({\tplan P}, B_{2 \varepsilon_0}(x)) + \ener^\alpha({\tplan P}, B_{2 \varepsilon_0}(y)).
\end{equation*}
\end{prp}

\begin{proof}

\textit{Step 1 \-- Choice of a suitable radius.}
Since ${\tplan P}$ is rectifiable, recalling \Cref{prop:slice}, we have:
\begin{multline*}
\frac{1}{\varepsilon_0} \int_{\varepsilon_0}^{2\varepsilon_0} \mass^\alpha(\islice{\tplan P,d_x,\eps}) + \mass^\alpha(\islice{\tplan P,d_y,\eps}) \dd{\varepsilon}\\
\le \frac{\ener^\alpha({\tplan P}, \{\varepsilon_0 \le d_x \le 2 \varepsilon_0\}) + \ener^\alpha({\tplan P}, \{\varepsilon_0 \le d_y \le 2 \varepsilon_0\})}{\varepsilon_0},
\end{multline*}
where $d_u : z \in \R^d \mapsto |z - u|$.
Therefore, there exists $\varepsilon \in [\varepsilon_0, 2 \varepsilon_0]$ such that
\begin{align} \label{eq:boundenergy}
\mass^\alpha(\islice{\tplan P,d_x,\eps}) + \mass^\alpha(\islice{\tplan P,d_y,\eps})
&\le \frac{\ener^\alpha({\tplan P}, \{\varepsilon_0 \le d_x \le 2 \varepsilon_0\}) + \ener^\alpha({\tplan P}, \{\varepsilon_0 \le d_y \le 2 \varepsilon_0\})}{\varepsilon_0} \nonumber \\
&\le \frac{\ener^\alpha({\tplan P}, B_{2 \varepsilon_0}(x)) + \ener^\alpha({\tplan P}, B_{2 \varepsilon_0}(y))}{\varepsilon_0}.
\end{align}
Since $\varepsilon \geq \eps_0$, we still have $\min\{{\tplan P}(\Gamma_\varepsilon(x, y)), {\tplan P}(\Gamma_\varepsilon(y, x))\} \ge m$.

\emph{Step 2 \-- Construction of the shortcut.}
Given $u \in \R^d$ and $\gamma \in \lipone$, we define:
\begin{equation*}
\begin{split}
&t_u^-(\gamma) := \inf \set{t \in [0, T_\infty(\gamma)] : \gamma(t) \in \partial B_\varepsilon(u)} \quad \text{ and } \\
&t_u^+(\gamma) := \sup \set{t \in  [0, T_\infty(\gamma)]  : \gamma(t) \in \partial B_\varepsilon(u)},
\end{split}
\end{equation*}
which belong to $[0, \infty]$, accepting the abuse of notation that $\inf \emptyset=0$ and $\sup \emptyset=0$. For any curve $\gamma \in \Gamma_\varepsilon(x, y)$ with $T_\infty(\gamma)<\infty$, we have that $t_x^-(\gamma)$ and $t_y^+(\gamma)$ belong to $[0, T_\infty(\gamma)]$ and satisfy $t_x^-(\gamma) < t_y^+(\gamma)$, given that $B_\varepsilon(x)$ and $B_\varepsilon(y)$ are disjoint (since $\varepsilon \le 2 \varepsilon_0 \le |y - x|/4$). Then, for any curve $\gamma \in \Gamma_\varepsilon(x, y)$ with $T_\infty(\gamma)<\infty$, we set
\begin{equation*}
\varphi_0^u(\gamma) := \gamma_{\vert [0, t_u^-(\gamma)]} \quad \text{and} \quad \varphi_\infty^u(\gamma) := \gamma_{\vert [t_u^+(\gamma), + \infty)}.
\end{equation*}
Lastly we consider arbitrary (non relabeled) measurable extensions of $\varphi_0^u$ and $\varphi_\infty^u$ to $\lipone$.
We also set:
\begin{align*}
\Lambda_\varepsilon(x, y) &:= \Gamma_\varepsilon(x, y) \setminus \Gamma_\varepsilon(y, x),
& \Lambda_\varepsilon(x, y, x) &:= \set{\gamma \in \Gamma_\varepsilon(x, y) \cap \Gamma_\varepsilon(y, x) : t_x^-(\gamma) < t_y^-(\gamma)}, \\
\Lambda_\varepsilon(y, x) &:= \Gamma_\varepsilon(y, x) \setminus \Gamma_\varepsilon(x, y),
& \Lambda_\varepsilon(y, x, y) &:= \set{\gamma \in \Gamma_\varepsilon(x, y) \cap \Gamma_\varepsilon(y, x) : t_y^-(\gamma) < t_x^-(\gamma)}.
\end{align*}
Defining
\begin{equation*}
m_{x} := \frac{\min\{{\tplan P}(\Lambda_\varepsilon(x, y)), {\tplan P}(\Lambda_\varepsilon(y, x))\}}{{\tplan P}(\Lambda_\varepsilon(x, y))} \quad \text{and} \quad m_{y} := \frac{\min\{{\tplan P}(\Lambda_\varepsilon(x, y)), {\tplan P}(\Lambda_\varepsilon(y, x))\}}{{\tplan P}(\Lambda_\varepsilon(y, x))}
\end{equation*}
with the convention $0/0 = 0$, we set:
\begin{align*}
\tplan Q_1 &:=
m_{x} {\tplan P} \mres \Lambda_\varepsilon(x, y)
- (\varphi_0^x)_\sharp m_{x} {\tplan P} \mres \Lambda_\varepsilon(x, y)
- (\varphi_\infty^y)_\sharp m_{x} {\tplan P} \mres \Lambda_\varepsilon(x, y) \\
&\qquad + m_{y} {\tplan P} \mres \Lambda_\varepsilon(y, x)
- (\varphi_0^y)_\sharp m_{y} {\tplan P} \mres \Lambda_\varepsilon(y, x)
- (\varphi_\infty^x)_\sharp m_{y} {\tplan P} \mres \Lambda_\varepsilon(y, x), \\
\tplan Q_2 &:=
{\tplan P} \mres \Lambda_\varepsilon(x, y, x)
- (\varphi_0^x)_\sharp {\tplan P} \mres \Lambda_\varepsilon(x, y, x)
- (\varphi_\infty^x)_\sharp {\tplan P} \mres \Lambda_\varepsilon(x, y, x) \\
&\qquad + {\tplan P} \mres \Lambda_\varepsilon(y, x, y)
- (\varphi_0^y)_\sharp {\tplan P} \mres \Lambda_\varepsilon(y, x, y)
- (\varphi_\infty^y)_\sharp {\tplan P} \mres \Lambda_\varepsilon(y, x, y).
\end{align*}
Furthermore, we define the traffic plan:
\begin{equation*}
\tilde{{\tplan P}} := {\tplan P} - \tplan Q_1 - \tplan Q_2.
\end{equation*}

We observe $\tilde{{\tplan P}}(\lipone) \le 3 {\tplan P}(\lipone)$. $\tilde{{\tplan P}}$ is supported on the set of simple curves and it is rectifiable. Moreover $\Sigma_{\tilde{{\tplan P}}}\subseteq \Sigma_{\tplan P}$.

One can easily check that
\begin{align*}
(e_\infty)_\sharp \tplan Q_1 - (e_0)_\sharp \tplan Q_1 &= \begin{multlined}[t]
(e_{t_y^+})_\sharp m_x {\tplan P} \mres \Lambda_\varepsilon(x, y) - (e_{t_x^-})_\sharp m_x {\tplan P} \mres \Lambda_\varepsilon(x, y)\\
+ (e_{t_x^+})_\sharp m_y {\tplan P} \mres \Lambda_\varepsilon(y, x) - (e_{t_y^-})_\sharp m_y {\tplan P} \mres \Lambda_\varepsilon(y, x),
\end{multlined}\\
\shortintertext{and}
(e_\infty)_\sharp \tplan Q_2 - (e_0)_\sharp \tplan Q_2 &= \begin{multlined}[t]
(e_{t_x^+})_\sharp {\tplan P} \mres \Lambda_\varepsilon(x, y, x) - (e_{t_x^-})_\sharp {\tplan P} \mres \Lambda_\varepsilon(x, y, x)\\
+ (e_{t_y^+})_\sharp {\tplan P} \mres \Lambda_\varepsilon(y, x, y) - (e_{t_y^-})_\sharp {\tplan P} \mres \Lambda_\varepsilon(y, x, y).
\end{multlined}
\end{align*}
Hence $(e_\infty)_\sharp \tplan Q_1 - (e_0)_\sharp \tplan Q_1 + (e_\infty)_\sharp \tplan Q_2 - (e_0)_\sharp \tplan Q_2 = S_\varepsilon^x + S_\varepsilon^y$ where we have grouped terms in $x$ and $y$, setting:
\begin{align*}
S_\varepsilon^x &:= \begin{multlined}[t]
(e_{t_x^+})_\sharp m_y {\tplan P} \mres \Lambda_\varepsilon(y, x) - (e_{t_x^-})_\sharp m_x {\tplan P} \mres \Lambda_\varepsilon(x, y)\\
+ (e_{t_x^+})_\sharp {\tplan P} \mres \Lambda_\varepsilon(x, y, x) - (e_{t_x^-})_\sharp {\tplan P} \mres \Lambda_\varepsilon(x, y, x),
\end{multlined}\\
S_\varepsilon^y &:= \begin{multlined}[t]
(e_{t_y^+})_\sharp m_x {\tplan P} \mres \Lambda_\varepsilon(x, y) - (e_{t_y^-})_\sharp m_y {\tplan P} \mres \Lambda_\varepsilon(y, x)\\
+ (e_{t_y^+})_\sharp {\tplan P} \mres \Lambda_\varepsilon(y, x, y) - (e_{t_y^-})_\sharp {\tplan P} \mres \Lambda_\varepsilon(y, x, y).
\end{multlined}
\end{align*}
Then we denote by $\tilde{T}$ the current induced by $\tilde{{\tplan P}}$. Since the boundary of $\tilde{T}$ is equal to $(e_\infty)_\sharp \tilde{{\tplan P}} - (e_0)_\sharp \tilde{{\tplan P}}$, this yields:
\begin{equation*}
\partial \tilde{T} = \mu^+ - \mu^- - (S_\varepsilon^x + S_\varepsilon^y).
\end{equation*}
Since $\tilde{T}$ does not irrigate the same measures as ${\tplan P}$, we adjust it by adding cones over $x$ and $y$ {(see \cite[26.26]{Sim})}:
\begin{equation*}
\bar T := \tilde{T} + x \cone S_\varepsilon^x + y \cone S_\varepsilon^y.
\end{equation*}
Since $S_\varepsilon^x(\R^d) = S_\varepsilon^y(\R^d) = 0$, then $\partial(x \cone S_\varepsilon^x) = S_\varepsilon^x$ and $\partial(y \cone S_\varepsilon^y) =  S_\varepsilon^y$. Hence we deduce that $\bar T \in \tpathset(\mu^-, \mu^+)$.

\emph{Step 3 \-- Energy estimate.}
Since $\tilde{{\tplan P}}$ is a traffic plan supported on the set of simple curves, we can write $\ener^\alpha(\tilde{{\tplan P}}) = \int_{\Sigma_{\tplan P}} \theta_{\tilde{{\tplan P}}}^\alpha \dd{\haus^1}$ and we find a bound on its $\alpha$-energy as follows:
\begin{align}\label{eq:enerestim1}
\ener^\alpha(\tilde{{\tplan P}}) &= \int_{\Sigma_{\tplan P}} \left( \theta_{\tplan P} - (\theta_{\tplan P} - \theta_{\tilde{{\tplan P}}}) \right)^\alpha \dd{\haus^1}\nonumber\\
&\le \int_{\Sigma_{\tplan P}} \left( \theta_{\tplan P}^\alpha - \alpha (\theta_{\tplan P} - \theta_{\tilde{{\tplan P}}}) \theta_{\tplan P}^{\alpha - 1} \right) \dd{\haus^1} \nonumber\\
&\le \int_{\Sigma_{\tplan P}} \theta_{\tplan P}^\alpha \dd{\haus^1} - \alpha {\tplan P}(\lipone)^{\alpha - 1} \int_{\Sigma_{\tplan P}} \left( \theta_{\tplan P} - \theta_{\tilde{{\tplan P}}} \right) \dd{\haus^1}\nonumber\\
&= \ener^\alpha({\tplan P}) - \alpha {\tplan P}(\lipone)^{\alpha - 1} \int_{\Sigma_{\tplan P}} \int_\lipone \one_{x \in \img \gamma} \dd(\tplan P - \tilde{\tplan P})(\gamma) \dd\haus^1(x) \nonumber\\
&= \ener^\alpha({\tplan P}) - \alpha {\tplan P}(\lipone)^{\alpha - 1} \int_\lipone \haus^1(\img \gamma) \dd(\tplan Q_1 + \tplan Q_2)(\gamma).
\end{align}
The first inequality follows from the concavity of $x \mapsto x^\alpha$ on $\R_+$, the second one is due to the fact that $\theta_{\tplan P} \le {\tplan P}(\lipone)$; then the equality consists in using the definition of the multiplicity and the fact that ${\tplan P}$ is supported on simple curves; and finally, Fubini's Theorem and the rectifiability of ${\tplan P}$ yield the last equality. In order to apply Fubini, one can first use the Hahn decomposition theorem to decompose the measure $\tplan Q_1 + \tplan Q_2$ in its positive and negative parts. Then one can apply Fubini on each of the two parts. 

In addition, we compute:
\begin{multline}\label{eq:enerestim2}
\int_\lipone \haus^1(\img \gamma) \dd{\tplan Q_1}(\gamma)\\
\begin{aligned}
 &= m_{x} \int_{\Lambda_\varepsilon(x, y)} \haus^1(\img \gamma) - \haus^1(\img \varphi_0^x(\gamma)) - \haus^1(\img \varphi_\infty^y(\gamma)) \dd \tplan P(\gamma)  \\
&\qquad + m_{y} \int_{\Lambda_\varepsilon(y, x)} \haus^1(\img \gamma) - \haus^1(\img \varphi_0^y(\gamma)) - \haus^1(\img \varphi_\infty^x(\gamma)) \dd \tplan P(\gamma)\\
&= m_{x} \int_{\Lambda_\varepsilon(x, y)} \haus^1\left(\img \gamma_{\vert [t_x^-(\gamma), t_y^+(\gamma)]}\right) \dd \tplan P(\gamma)\\
&\qquad+ m_{y} \int_{\Lambda_\varepsilon(y, x)} \haus^1\left(\img \gamma_{\vert [t_y^-(\gamma), t_x^+(\gamma)]}\right) \dd \tplan P(\gamma)\\
&\ge 2 (|y - x| - 2 \varepsilon) \min \{ {\tplan P}(\Lambda_\varepsilon(x, y)), {\tplan P}(\Lambda_\varepsilon(y, x)) \} \\
&\ge |y - x| \min \{ {\tplan P}(\Lambda_\varepsilon(x, y)), {\tplan P}(\Lambda_\varepsilon(y, x)) \}.
\end{aligned}
\end{multline}
The first equality is straightforward, the second is a consequence of the fact that ${\tplan P}$ is supported on the set of simple curves, then the inequality comes from $\gamma(t_x^-(\gamma)) \in \bar B_\eps(x)$ and $\gamma(t_y^+(\gamma)) \in\bar B_\eps(y)$, and the final inequality results from the assumption on $\varepsilon$. Similarly:
\begin{multline}\label{eq:enerestim3}
\int_\lipone \haus^1(\img \gamma) \dd{\tplan Q_2}(\gamma)\\
\begin{aligned}
 &= \begin{multlined}[t]\int_{\Lambda_\varepsilon(x, y, x)} \haus^1(\img \gamma) - \haus^1(\img \varphi_0^x(\gamma)) - \haus^1(\img \varphi_\infty^x(\gamma)) \dd \tplan P(\gamma)  \\
+ \int_{\Lambda_\varepsilon(y, x, y)} \haus^1(\img \gamma) - \haus^1(\img \varphi_0^y(\gamma)) - \haus^1(\img \varphi_\infty^y(\gamma)) \dd \tplan P(\gamma)
\end{multlined}\\
&= \int_{\Lambda_\varepsilon(x, y, x)} \haus^1\left(\img \gamma_{\vert [t_x^-(\gamma), t_x^+(\gamma)]}\right) \dd \tplan P(\gamma) + \int_{\Lambda_\varepsilon(y, x, y)} \haus^1\left(\img \gamma_{\vert [t_y^-(\gamma), t_y^+(\gamma)]}\right) \dd \tplan P(\gamma) \\
&\ge 2 (|y - x| - 2 \varepsilon) \left( {\tplan P}(\Lambda_\varepsilon(x, y, x)) + {\tplan P}(\Lambda_\varepsilon(y, x, y)) \right) \\
&\ge |y - x| {\tplan P}(\Gamma_\varepsilon(x, y) \cap \Gamma_\varepsilon(y, x)).
\end{aligned}
\end{multline}
Combining \Cref{eq:enerestim1,eq:enerestim2,eq:enerestim3} yields the following bound:
\begin{align*}
\ener^\alpha(\tilde{{\tplan P}}) &\le \ener^\alpha({\tplan P}) - \alpha {\tplan P}(\lipone)^{\alpha - 1} |y - x| ( \min \{ {\tplan P}(\Lambda_\varepsilon(x, y)), {\tplan P}(\Lambda_\varepsilon(y, x)) \}\\
&\qquad+ {\tplan P}(\Gamma_\varepsilon(x, y) \cap \Gamma_\varepsilon(y, x)))\\
&\le \ener^\alpha({\tplan P}) - \alpha {\tplan P}(\lipone)^{\alpha - 1} |y - x| \min \{ {\tplan P}(\Gamma_\varepsilon(x, y)), {\tplan P}(\Gamma_\varepsilon(y, x)) \} \\
&= \ener^\alpha({\tplan P}) - \alpha {\tplan P}(\lipone)^{\alpha - 1} |y - x| m.
\end{align*}

Moreover, by definition of $t_x^\pm$, $S_\varepsilon^x$ is supported on the sphere centered at $x$ with radius $\varepsilon$. Therefore $\mass^\alpha(x \cone S_\varepsilon^x) \le \varepsilon \mass^\alpha(S_\varepsilon^x)$. In addition, note that we have by definition $\mass^\alpha(S_\varepsilon^x) \le \mass^\alpha(\islice{\tplan P,d_x,\eps})$. The same computations also hold at $y$. Thus we find that $\partial \bar T = \mu^+ - \mu^-$ and by subadditivity of the $\alpha$-mass
\begin{align*}
\mass^\alpha(\bar T)
&\leq \mass^\alpha(\tilde T) + \varepsilon \left( \mass^\alpha(S_\varepsilon^x) + \mass^\alpha(S_\varepsilon^y) \right) \\
&\le \ener^\alpha({\tplan P}) - \alpha {\tplan P}(\lipone)^{\alpha - 1} |y - x| m + \varepsilon \left( \mass^\alpha(S_\varepsilon^x) + \mass^\alpha(S_\varepsilon^y) \right) \\
&\le \ener^\alpha({\tplan P}) - \alpha {\tplan P}(\lipone)^{\alpha - 1} |y - x| m + \varepsilon \left( \mass^\alpha(\islice{\tplan P,d_x,\eps}) + \mass^\alpha(\islice{\tplan P,d_y,\eps}) \right) \\
\alignedrefabove{\cref{eq:boundenergy}}{\le} \ener^\alpha({\tplan P}) - \alpha {\tplan P}(\lipone)^{\alpha - 1} |y - x| m + \ener^\alpha({\tplan P}, B_{2 \varepsilon_0}(x)) + \ener^\alpha({\tplan P}, B_{2 \varepsilon_0}(y)).\qedhere
\end{align*}
\end{proof}

\section{Proof of the main theorem}\label{sec:mainproof}

\begin{proof}[Proof of \Cref{mainthm}]
Take $\{\tplan P_n\}_{n\in\N}$ and $\tplan P$ satisfying the hypotheses of the theorem. First, we know that $\tplan P \in \tplanset(\mu^-,\mu^+)$, since the evaluation maps $e_0,e_\infty$ are continuous on $\tplanset_C \coloneqq \set{\tplan P \in \tplanset : \int_\lipone T_\infty \dd\tplan P \leq C}$ which is a closed subset of $\tplanset$, where $C \coloneqq \sup_n \int_\lipone T_\infty \dd\tplan P_n$ (see \cite[Section~1]{Peg17}). Set $\{T_n\}_{n\in\N}$ and $T$ to be the transport paths induced by $\{\tplan P_n\}_{n\in\N}$ and $\tplan P$. We first show that $\{T_n\}$ and $T$ satisfy the hypotheses of \cite[Theorem~1.1]{ColDeRMar19-2} so that $T$ is optimal; then we prove that $\tplan P$ is a good decomposition of $T$ by showing successively \cref{prop:C}, then \cref{prop:B} (which is the heart of the proof), and finally \cref{prop:A} of \Cref{gooddecompo}. From this we conclude that $\tplan P$ is optimal.

\textit{Step~1 \-- $T$ is optimal.} Let us prove that $T_n$ converges weakly-$\star$ to $T$. Let $\omega \in C^0(X,\R^d)$ and fix $\eps > 0$. By Markov's inequality, since all $T_n$'s and $T$ lie in $\tplanset_C$, there exists $t_0 \geq 0$ such that for all $n \in \N$, ${\tplan P}_n(T_\infty \ge t_0) \le \eps /3 \norm{\omega}_\infty$ and $\tplan P(T_\infty \ge t_0) \le \eps /3 \norm{\omega}_\infty$. The map $f_{t_0} : \gamma \mapsto \int_0^{t_0} \omega(\gamma(t)) \cdot \gamma'(t) \dd{t}$ is continuous on $\lipone$: if $\gamma_n \to \gamma$ then $\omega \circ \gamma_n \to \omega\circ \gamma$ strongly in $L^1([0,t_0])$ by the Dominated Convergence Theorem, and $\{|\gamma'_n|\}_{n \in \N}$ is bounded in $L^\infty([0, t_0])$ hence $f_{t_0}(\gamma_n) \to f_{t_0}(\gamma)$. Thanks to the weak-$\star$ convergence of ${\tplan P}_n$ to ${\tplan P}$, one has for $n$ large enough:
\[
\abs*{\langle T_n, \omega \rangle - \langle T, \omega \rangle} \le \abs*{\int_{T_\infty < t_0} f_{t_0}(\gamma) \dd{({\tplan P}_n - {\tplan P})}(\gamma)} + \frac{2\varepsilon}{3} \le \varepsilon,
\]
thus $\langle T_n, \omega \rangle \to \langle T, \omega \rangle$ as $\eps$ is arbitrary. By equivalence of the Lagrangian and Eulerian models (\cite[Theorem 2.4.1]{PegPhD} or \cite{PaoSte06}), an optimal traffic plan induces an optimal transport path, hence $T_n$ is optimal since ${\tplan P}_n$ is, and $\mass^\alpha(T_n) = \ener^\alpha(\tplan P_n)$ so the energy is uniformly bounded and we may apply the Eulerian stability result of \cite[Theorem~1.1]{ColDeRMar19-2} to conclude that $T \in \otpathset(\mu^-,\mu^+)$. 

\textit{Step~2 \-- Proof of \cref{prop:C}.} We only use the fact that $\mu^-$ and $\mu^+$ are mutually singular. Denote by $\Gamma$ the set of eventually constant curves of $\lipone$ such that $e_0(\gamma) = e_\infty(\gamma)$. We obviously have $(e_0)_\sharp {\tplan P} \mres \Gamma = (e_\infty)_\sharp {\tplan P} \mres \Gamma$, and
\begin{equation*}
(e_0)_\sharp {\tplan P} \mres \Gamma \le (e_0)_\sharp {\tplan P} = \mu^-, \quad (e_\infty)_\sharp {\tplan P} \mres \Gamma \le (e_\infty)_\sharp {\tplan P} = \mu^+.
\end{equation*}
Yet $\mu^-$ and $\mu^+$ are mutually singular, hence ${\tplan P}(\Gamma) = 0$. In particular, ${\tplan P}$-almost every curve $\gamma$ is non-constant and $e_0(\gamma) \ne e_\infty(\gamma)$, hence $\mass(\partial I_\gamma) = 2$. The fact that $\mu^-$ and $\mu^+$ are mutually singular allows to write $\mass(\partial T) = \mu^-(\R^d) + \mu^+(\R^d)$. Therefore
\begin{equation*}
\mass(\partial T) = 2 {\tplan P}(\lipone) = \int_\lipone \mass(\partial I_\gamma) \dd \tplan P(\gamma).
\end{equation*}

\textit{Step~3 \-- Proof of \cref{prop:B}.} It will result from the fact that $\haus^1( \{ \bar\theta_{\tplan P} > 0 \} ) = 0$, where we recall $\bar\theta_{\tplan P}$ is defined in \cref{eq:bartheta}. We argue by contradiction, assuming that $\haus^1( \{ \bar\theta_{\tplan P} > 0 \} ) > 0$, and we wish to show that this contradicts the optimality of a $T_n$ for some large $n$. 
Using \Cref{lem:concentration} and \Cref{thm:cycles}, take $x_0 \in \{ \bar \theta_{\tplan P} > 0 \}$ such that $\liminf_{n \to \infty} \mass^\alpha(T_n \mres B_\varepsilon(x_0)) = o(1)$
as $\varepsilon \to 0$, together with a set $G$ satisfying the conclusion of \Cref{thm:cycles}. In particular, $\haus^1(G) > 0$. By monotone convergence, there exists $r > 0$ such that the set $G' := G \cap \R^d \setminus B_r(x_0)$ has positive $\haus^1$-measure. Now let $0 < \tilde\eps \leq r/8$ be such that:
\begin{equation*}
\liminf_{n \to \infty} \mass^\alpha(T_n \mres B_{2 \tilde\eps}(x_0)) \le \alpha {\tplan P}(\lipone)^{\alpha - 1} r \frac{\bar \theta (x_0)}{128}.
\end{equation*}
There exists a subsequence $\{T_{\tilde n_k}\}_{k\in\N}\subseteq \{T_{ n}\}_{n\in\N}$ such that
\begin{equation} \label{eq:energyvanishing}
\forall k \in \N, \quad \mass^\alpha(T_{\tilde n_k} \mres B_{2 \tilde\eps}(x_0)) \le \alpha {\tplan P}(\lipone)^{\alpha - 1} r \frac{\bar \theta (x_0)}{64}.
\end{equation}
Then we choose $x \in G'$ satisfying the conclusion of \Cref{lem:concentration} for the subsequence $\{T_{\tilde n_k}\}_{k\in\N}$. Take $\eps \leq \tilde\eps$ such that for a further subsequence $\{T_{n_k}\}_{k\in\N}$, \cref{eq:energyvanishing} holds with $x,\eps, \{T_{n_k}\}$ in place of $x_0, \tilde \eps, \{T_{\tilde n_k}\}$. By monotonicity in $\eps$, notice that \cref{eq:energyvanishing} also holds with $x_0, \eps, \{T_{n_k}\}$. To simplify notation, the subsequence $\{T_{n_k}\}$ will just be denoted by $\{T_n\}$.  
Since
\[\min\{\tplan P(\Gamma(x,x_0), \tplan P(\Gamma(x,x_0)\} \geq \frac{\bar\theta(x_0)}4,\] 
we know by \Cref{lem:quasicycles} that there exists $N \in \N$ such that for any $n \ge N$,
\[\min \left\{{\tplan P}_n(\Gamma_\varepsilon(x_0, x)), {\tplan P}_n(\Gamma_\varepsilon(x, x_0))\right\} \ge \frac{\bar \theta(x_0)}8.\]
Moreover, notice that ${\tplan P}_n \weakstarto {\tplan P}$ implies ${\tplan P}_n(\lipone) \to {\tplan P}(\lipone)$ (because $\lipone$ is compact), so up to increasing $N$, one may assume ${\tplan P}_n(\lipone)^{\alpha - 1} \ge {\tplan P}(\lipone)^{\alpha - 1}/2$ for all $n \ge N$. Since $\varepsilon \le r/8 \le |x - x_0|/8$, we can apply \Cref{prop:comp} to ${\tplan P}_N$. Thus there exists a transport path $\bar T$ connecting $\mu_N^-$ to $\mu_N^+$ satisfying:
\begin{align*}
\mass^\alpha(\bar T) &\le \ener^\alpha({\tplan P}_{N}) - \alpha {\tplan P}_{N}(\lipone)^{\alpha - 1} \frac{\bar\theta(x_0)}{8} |x - x_0| + \alpha {\tplan P}(\lipone)^{\alpha - 1} r \frac{\bar \theta (x_0)}{32} \\
&\le \ener^\alpha({\tplan P}_{N}) - \alpha {\tplan P}(\lipone)^{\alpha - 1} |x - x_0| \frac{\bar \theta (x_0)}{16} + \alpha {\tplan P}(\lipone)^{\alpha - 1} |x - x_0| \frac{\bar \theta (x_0)}{32} \\
&\le \ener^\alpha({\tplan P}_{N}) - \alpha {\tplan P}(\lipone)^{\alpha - 1} |x - x_0| \frac{\bar \theta (x_0)}{32} \\
&< \ener^\alpha({\tplan P}_{N}),
\end{align*}
which contradicts the optimality of ${\tplan P}_{N}$. Hence $\bar\theta_{\tplan P}(x) = 0$ for $\haus^1$-a.e. $x \in \Sigma_{\tplan P}$ which by \Cref{cancelcharac}, is equivalent to $\abs{\vec \theta_{\tplan P}(x)} = \Theta_{\tplan P}(x)$. As a consequence, since $\tplan P$ is rectifiable, we have equality almost everywhere in \cref{eq:strict1}, which is equivalent to \cref{prop:B} by \Cref{ineq_good_decompo}.

\textit{Step~4 \-- Proof of \cref{prop:A}} This item follows from the absence of cancellations in ${\tplan P}$ and from the fact that $T$ is acyclic as an optimal transport path, as observed in \cite[Theorem~10.1]{PaoSte06} (notice that \cite[Theorem~1.1]{ColDeRMar19-2} is also essential here). We have already noticed in \text{Step~2} that $\tplan P$-a.e. curve is nonconstant, thus it remains to show that almost every curve is simple. Denote by $\Gamma$ the set of curves that are eventually constant but not simple. For any $\gamma \in \Gamma$, there exist $s < t$ such that $\gamma(s) = \gamma(t)$ and $\gamma$ is non-constant on $[s, t]$, i.e. $\gamma_{|[s,t]}$ is a nontrivial loop. Let $r : \Gamma \to \lipone$ a map that associates to each $\gamma \in \Gamma$ a nontrivial loop $\gamma_{|[s(\gamma),t(\gamma)]}$. Note that one can build $r$ to be Borel: for example, one can check there exists a finite number of loops with maximal length and take the first one. Then for any $\gamma \in \Gamma$, $I_{r(\gamma)}$ is a cycle (in the sense of currents), that is to say $\partial I_{r(\gamma)} = 0$. Denote by $S$ the current induced by the traffic plan $r_\sharp {\tplan P}$: it is obviously a cycle.

If $x \in\Sigma_{\tplan P}$ is a regular point for $\tplan P$, since $\tplan P$ has no cancellation we know by \Cref{cancelcharac} that there exists $s\in \{-1,+1\}$ such that $\gamma'(t) = s\abs{\gamma'(t)}\tau_{\Sigma_{\tplan P}}(x)$ for every $t\in \gamma^{-1}(x)$ and ${\tplan P}$-almost every curve $\gamma$, which implies also:
\begin{equation}\begin{split}
\abs{\vec \theta_{r_\sharp \tplan P}(x)} &= \abs*{\int_\Gamma \vec m_{r(\gamma)}(x) \dd \tplan P(\gamma)} = \int_\Gamma \# r(\gamma)^{-1}(x) \dd \tplan P(\gamma),\\
&\leq \int_\Gamma \# \gamma^{-1}(x) \dd \tplan P(\gamma)= \abs*{\int_\Gamma \vec m_{\gamma}(x) \dd \tplan P(\gamma)} = \abs{\vec \theta_{\tplan P}(x)}.
\end{split}
\end{equation}
Because they are all positively colinear, we get:
\begin{equation}
\abs*{\vec \theta_{\tplan P}(x)} = \abs{\vec \theta_{r_\sharp \tplan P}(x)} + \abs{\vec \theta_{\tplan P}(x)-\vec \theta_{r_\sharp \tplan P}(x)},
\end{equation}
and knowing that $T = \rectcurr{\Sigma_{\tplan P},\vec \theta_{\tplan P}(x)}$ and $S = \rectcurr{\Sigma_{\tplan P},\vec \theta_{r_\sharp \tplan P}(x)}$, integrating over $\Sigma_{\tplan P}$ yields:
\begin{equation}
\mass(T) = \mass(S) + \mass(T-S).
\end{equation}
However, $T$ is acyclic as an optimal transport path, thus $S = 0$. Yet by Fubini's Theorem and the Area Formula:
\begin{equation*}
0 = \mass(S) = \int_{\Sigma_{\tplan P}} \int_\Gamma \# r(\gamma)^{-1}(x) \dd \tplan P(\gamma) \dd{\haus^1}(x)
= \int_\Gamma \text{length} \, r(\gamma) \dd \tplan P(\gamma),
\end{equation*}
from which we deduce ${\tplan P}(\Gamma) = 0$, since $\text{length} \, r(\gamma) > 0$ for every $\gamma \in \Gamma$.

\textit{Step~5 \-- $\tplan P$ is optimal.} Let us conclude. By \cref{prop:B}, $\abs{\vec \theta_{\tplan P}(x)} = \int_\lipone \abs{\vec m_\gamma(x)} \dd\tplan P(x)$, and by \cref{prop:A}, $\int_\lipone \abs{\vec m_\gamma(x)} \dd\tplan P(x) = \theta_{\tplan P}(x) = \Theta_{\tplan P}(x)$ for $\haus^1$-a.e. $x$, hence:
\[\ener^\alpha(\tplan P) = \int_{\Sigma_{\tplan P}} \theta_{\tplan P}^{\alpha -1} \Theta_{\tplan P} \dd\haus^1 = \int_{\Sigma_{\tplan P}} \abs{\vec \theta_{\tplan P}}^\alpha \dd\haus^1 = \mass^\alpha(T).\]
Since $T$ is optimal, by equivalence of the Lagrangian and Eulerian models we know that the optimal costs are the same and we get that $\tplan P \in \otplanset(\mu^-,\mu^+)$.
\end{proof}

\subsection*{Acknowledgments}
	M. C. was partially supported by the Swiss National Science Foundation grant 200021\_\\182565. A. D. R. has been supported by the NSF DMS Grant No.~1906451. A. M. acknowledges partial support from GNAMPA-INdAM. P. P. wishes to thank M. C. and the AMCV Chair for the invitation for a semester at EPFL during which this research was partly conducted. A. P. acknowledges support from the Fondation Mathématique Jacques Hadamard and from the AMCV chair at EPFL, and thanks M. C. and P. P. for their supervision of his Master 1 thesis.

\printbibliography

\noindent%
Maria Colombo
\\
EPFL SB, Station 8, 
CH-1015 Lausanne, Switzerland
\\
e-mail M.C.: {\texttt maria.colombo@epfl.ch}
\\
~
\\
Antonio De Rosa
\\
Courant Institute of Mathematical Sciences, New York University, New York, NY, USA
\\
e-mail A.D.R.: {\texttt derosa@cims.nyu.edu}
\\
~
\\
Andrea Marchese
\\
Dipartimento di Matematica Via Sommarive, 14 - 38123 Povo - Italy\\
e-mail A.M.: {\texttt andrea.marchese@unitn.it}
\\
~
\\
Paul Pegon
\\
Universit\'e Paris-Dauphine, PSL Research University, Ceremade, INRIA, Project team Mokaplan, France\\
e-mail P.P.: {\texttt pegon@ceremade.dauphine.fr}
\\
~
\\
Antoine Prouff
\\
DER de Mathématiques, ENS Paris-Saclay, Université Paris-Saclay, France\\
e-mail A.P.: {\texttt antoine.prouff@ens-paris-saclay.fr}

\end{document}